\documentclass[12pt]{smfart}
\usepackage{color}
\usepackage{amssymb,verbatim}
\usepackage{amsthm,array,amssymb,amscd,amsfonts,amssymb,latexsym, url}
\usepackage{amsmath}
\usepackage[all]{xy}
\usepackage[french]{babel}
\usepackage{times}

\newcounter{spec}
{\end{list}}

\renewcommand{\P}{{\mathbf P}}

\newcommand{\R}{{\mathbb R}}
\newcommand{\Z}{{\mathbb Z}}
\newcommand{\Q}{{\mathbb Q}}

\renewcommand{\L}{{\mathbb L}}

\newcommand{\oi}{\hskip1mm {\buildrel \simeq \over \rightarrow} \hskip1mm}

\newcommand{\Br}{{\operatorname{Br}}}

\newcommand{\Hom}{{\operatorname{Hom}}}
\newcommand{\Spec}{{\operatorname{Spec \ }}}

\renewcommand{\lim}{\varprojlim}

\renewcommand{\phi}{\varphi}

\numberwithin{equation}{section}

\newfont{\gothic}{eufb10}

\renewcommand{\qed}{{\hfill$\square$}}

\newtheorem{theo}{Th\'{e}or\`{e}me}[section]
\newtheorem{prop}[theo]{Proposition}

\newtheorem{lem}[theo]{Lemme}
\newtheorem{cor}[theo]{Corollaire}
\theoremstyle{definition}
\newtheorem{defi}[theo]{D\'efinition}
\theoremstyle{remark}

\newtheorem{rem}[equation]{Remarque}

\newcommand{\bthe}{\begin{theo}}
\newcommand{\ble}{\begin{lem}}
\newcommand{\bpr}{\begin{prop}}
\newcommand{\bco}{\begin{cor}}
\newcommand{\bde}{\begin{defi}}
\newcommand{\ethe}{\end{theo}}
\newcommand{\ele}{\end{lem}}
\newcommand{\epr}{\end{prop}}
\newcommand{\eco}{\end{cor}}
\newcommand{\ede}{\end{defi}}

\newcommand{\Gal}{{\rm{Gal}}}
\newcommand{\et}{   {\operatorname{\acute{e}t}   }   }

\newcommand{\Pic}{\operatorname{Pic}}

\DeclareFontFamily{U}{wncy}{}
\DeclareFontShape{U}{wncy}{m}{n}{%
<5>wncyr5%
<6>wncyr6%
<7>wncyr7%
<8>wncyr8%
<9>wncyr9%
<10>wncyr10%
<11>wncyr10%
<12>wncyr6%
<14>wncyr7%
<17>wncyr8%
<20>wncyr10%
<25>wncyr10}{}
\DeclareMathAlphabet{\cyr}{U}{wncy}{m}{n}
\def\A{{\mathbb A}}
\def\pic{{\rm Pic\,}}
\def\G{{\mathbb G}}
\def\k{{\overline k}}
\def\X{{\mathcal X}}
\def \spec {{\rm Spec}\, }
\def \ov{\overline}
\def \RR{{\bf R}}
\def \cores{{\rm Cores \,}}
\def \Ga{{\Gamma}}
\DeclareMathOperator{\Ext}{Ext}
\def \calo{{\mathcal O}}
\def\U{{\mathcal U}}
\def\W{{\mathcal W}}

\begin{document}
%-----------------------------------------------------------------------%

\title[Approximation forte en famille]{Approximation forte en famille}
\author{Jean-Louis Colliot-Th\'el\`ene}
\address{C.N.R.S., Universit\'e Paris Sud\\Math\'ematiques, B\^atiment 425\\91405 Orsay Cedex\\France}
\email{jlct@math.u-psud.fr}
\author{David Harari}
\address{Universit\'e Paris Sud\\Math\'ematiques, B\^atiment 425\\91405 Orsay Cedex\\France}
\email{David.Harari@math.u-psud.fr}
\date{version r\'evis\'ee, 14 juillet 2013}
\maketitle

 \begin{abstract}
 Soient $k$ un corps de nombres et $X$ une $k$-vari\'et\'e affine lisse int\`egre 
  fibr\'ee au-dessus de la droite affine $\A^1_{k}$. Supposons que toutes les fibres sont g\'eom\'etriquement int\`egres, et que la fibre g\'en\'erique est un espace homog\`ene sous un
  groupe semisimple simplement connexe presque simple $G/k(t)$, les stabilisateurs g\'eom\'etriques \'etant r\'eductifs connexes.
  Soit $v$ une place de $k$ telle que la fibration $X \to \A^1_{k}$ admette une section rationnelle sur le compl\'et\'e
  $k_{v}$. Supposons en outre que pour presque tout point   $x \in \A^1(k_{v})$ le $k_{v}$-groupe $G_{x}$ soit isotrope.  Supposons enfin le groupe de Brauer de $X$ r\'eduit \`a celui de $k$. Alors l'approximation forte vaut pour $X$
  en dehors de la place $v$.
 \end{abstract}
 
 \begin{altabstract}
 Let $k$ be a number field and $X$ a smooth integral affine variety  equipped  with a surjective morphism  $f : X \to \A^1_{k}$  to the affine line. Assume that all fibres of $f$ are split, for instance that they are geometrically integral. Assume that the generic fibre of $f$ is a homogeneous space of a simply connected, almost simple, semisimple group $G/k(t)$, and that the geometric stabilizers are connected reductive groups. Let $v$ be a place of $k$ such that the fibration $f$ acquires a rational section over the completion $k_{v}$ at $v$. Assume moreover that at almost all points in $x \in \A^1(k_{v})$ the specialized group $G_{x}$ is isotropic over $k_{v}$. If the Brauer group of $X$ is reduced to the Brauer group of $k$, then strong approximation holds for $X$ away from the place $v$.
 \end{altabstract}

\section{Introduction}

On dit qu'une vari\'et\'e alg\'ebrique $X$ d\'efinie sur un corps de nombres $k$ et qui
poss\`ede des points dans tous les compl\'et\'es de $k$ satisfait l'approximation forte en dehors
d'un ensemble fini $S$ de places du corps $k$ si l'image diagonale de l'ensemble $X(k)$
des points rationnels de $X$ est dense dans l'espace $X(\A_{k}^S)$ des points ad\'eliques de $X$  hors de $S$.

Lorsqu'une telle propri\'et\'e vaut, elle implique pour tout mod\`ele de $X$ au-dessus
de l'anneau des $S$-entiers  de $k$ un principe local-global pour l'existence de points
$S$-entiers.

Pour $S$ non vide, l'approximation forte est une propri\'et\'e bien connue des espaces affines. 
Cette propri\'et\'e a \'et\'e \'etablie par M. Kneser et V. P. Platonov pour tout $k$-groupe
semisimple simplement connexe presque $k$-simple, sous l'hypoth\`ese que
le produit
des points locaux de $G$ aux places de 
$S$ 
est non compact.

Il a \'et\'e observ\'e que cette propri\'et\'e ne s'\'etend pas aux groupes non simplement connexes.
Cependant, une obstruction de type Brauer-Manin a \'et\'e d\'egag\'ee \cite{CTX} et il a \'et\'e
montr\'e  dans une s\'erie d'articles (voir  \cite{borodemarche}) que sous des hypoth\`eses tr\`es larges elle contr\^ole le d\'efaut d'approximation forte
pour les espaces homog\`enes de groupes lin\'eaires connexes.

Tout comme cela est entrepris pour le probl\`eme analogue pour les points rationnels, on souhaite \'elargir la
classe des vari\'et\'es pour lesquelles on a un tel contr\^ole sur les points entiers. L'obstruction de Brauer-Manin enti\`ere est en effet
 d\'efinie pour toute vari\'et\'e  \cite{CTX2}. Elle a \'et\'e calcul\'ee pour quelques vari\'et\'es qui ne sont pas
des espaces homog\`enes \cite{KT,CTW}, mais sans que l'on puisse dire s'il s'agit l\`a de la seule obstruction.

Dans \cite{CTX2},   Fei Xu et le premier auteur \'etudient  l'approximation forte pour
certains  mod\`eles lisses $X$ des vari\'et\'es d\'efinies sur un corps de nombres $k$ par une \'equation
$$q(x,y,z)=p(t),$$ o\`u $q(x,y,z)$ est une forme quadratique non d\'eg\'en\'er\'ee  
\`a coefficients dans un corps de nombres $k$ et $p(t) \in k[t]$ est un polyn\^ome non nul en une variable.
S'il existe une place $v_{0}$ telle que $q(x,y,z)$ est isotrope sur le compl\'et\'e $k_{v_{0}}$,
on a \'etabli que  l'image diagonale de $X(k)$ dans la projection de l'ensemble de Brauer--Manin 
$X(\A_k)^{\Br X}$
sur $X(\A_k^{v_0})$ (ad\`eles hors de $v_{0}$) est dense.

\medskip
Inspir\'es par cet exemple, 
  nous \'etablissons un r\'esultat g\'en\'eral pour
les familles \`a un param\`etre d'espaces homog\`enes. 
Nous montrons :

\medskip
Th\'eor\`eme A (voir le th\'eor\`eme \ref{thprincipal}).  {\it 
  Soit $X$ une $k$-vari\'et\'e lisse connexe 
munie d'un morphisme $f : X \to \A^1_{k} $ 
satisfaisant les conditions suivantes :

 (i) la fibre g\'en\'erique de $f$ est un espace homog\`ene d'un
$k(t)$-groupe semisimple simplement connexe presque simple $G_{t}$,
et les stabilisateurs g\'eom\'etriques pour cette action sont  r\'eductifs connexes;

(ii)  les fibres de $f$ sont scind\'ees, par exemple g\'eom\'etriquement int\`egres;

(iii)  il existe une place $v_{0}$ telle que la fibration $f$ ait une section rationnelle sur $k_{v_{0}}$,
 et que, pour presque tout $t_{0} \in k_{v_{0}}$ le groupe sp\'ecialis\'e $G_{t_{0}}$
est isotrope.

  (iv) Les \'el\'ements du groupe de Brauer $\Br X$
prennent une valeur constante lorsqu'on les \'evalue sur $X(k_{v_{0}})$.

Alors 
  l'image diagonale de $X(k)$ dans la projection de l'ensemble de Brauer--Manin 
$X(\A_k)^{\Br X}$ sur $X(\A_k^{v_0})$ est dense.}

\medskip

Sous les hypoth\`eses 
(i) et (ii)
du th\'eor\`eme, le quotient $\Br X/\Br k$ est fini.

\smallskip

Comme nous l'a fait remarquer Dasheng Wei, m\^eme pour $v_0$  r\'eelle,
on ne peut pas esp\'erer 
que la conclusion du th\'eor\`eme soit satisfaite 
si la projection sur $\A^1(k_{v_0})$ de $X(\A_k)^{\Br X}$
est compacte; il est donc clair qu'il faut une condition en $v_0$, condition 
qui est ici assur\'ee par la combinaison de nos hypoth\`eses (iii) et (iv).
Noter qu'en particulier l'hypoth\`ese (iv) est automatiquement
satisfaite dans plusieurs cas int\'eressants, voir la Remarque \ref{souventsatisfaite}. 
Les hypoth\`eses (iii) et (iv) sont toujours satisfaites si $v_0$ est complexe.

\smallskip

L'exemple suivant repr\'esente d\'ej\`a une vaste g\'en\'eralisation 
des principaux  r\'esultats de \cite{CTX2}. 
 
 \medskip

Th\'eor\`eme B (th\'eor\`eme \ref{generalCTX2}).  {\it Soient $a_{i}(t), i=1,2,3,$ et $p(t)$
des polyn\^{o}mes deux \`a deux premiers entre eux.
Soit $X\subset  \A^4_{k}$ l'ouvert de lissit\'e de la $k$-vari\'et\'e affine
$Y$ d'\'equation
$$\sum_{i=0}^2 a_{i}(t)x_{i}^2=p(t).$$
Soit $v_{0}$ une place de $k$.  On suppose que
la conique d'\'equation $\sum_{i=0}^2 a_{i}(t)x_{i}^2=0$ sur le corps $k(t)$ 
a un point rationnel sur le corps $k_{v_{0}}(t)$.

Si le produit $p(t).\prod_{i}a_{i}(t)$
est un polyn\^ome non constant s\'eparable, alors $X=Y$ et
 l'approximation forte vaut pour $X$ hors de $v_{0}$ : l'image diagonale de $X(k)$
est dense dans  $X(\A_k^{v_0})$ (ad\`eles hors de $v_{0}$).}

\smallskip

Quand $v_0$ est complexe, on a l'\'enonc\'e plus g\'en\'eral que
l'image diagonale de $X(k)$ dans la projection de l'ensemble de Brauer--Manin
$X(\A_k)^{\Br X}$
sur $X(\A_k^{v_0})$ (ad\`eles hors de $v_{0}$) est dense sans avoir besoin
de l'hypoth\`ese suppl\'ementaire sur le produit $p(t).\prod_{i}a_{i}(t)$.

\medskip

\`A titre d'illustration, voici un cas particulier.

\smallskip

Th\'eor\`eme C. {\it Soient $a_{i}(t), i=1,2,3,$ et $p(t)$ dans $\Z[t]$ des polyn\^{o}mes.
Supposons le produit $p(t).\prod_{i}a_{i}(t)$ non constant et sans facteur carr\'e dans $\Q[t]$.
Soit $\X/\Z$ le sch\'ema affine d\'efini dans $\A^4_{\Z}$ par
$$\sum_{i=0}^2 a_{i}(t)x_{i}^2=p(t).$$ 
Supposons que pour presque tout $t \in \R$
la conique $\sum_{i=0}^2 a_{i}(t)x_{i}^2=0$ a un point dans $\R$.
Alors le principe local-global et l'approximation forte valent pour les points entiers de  $\X$ :
L'image diagonale de $\X(\Z)$ est dense dans le produit $\prod_{p} \X(\Z_{p})$
des solutions locales enti\`eres sur tous les premiers $p$.}

\medskip
 
 Le lecteur peut se demander pourquoi dans les deux \'enonc\'es pr\'ec\'edents  on consid\`ere  une forme quadratique \`a trois variables
 dans le membre de gauche de l'\'equation.
 
  Lorsque l'on prend au moins 4 variables, sous l'hypoth\`ese
 que le produit $p(t).\prod_{i}a_{i}(t)$ est non constant et sans facteur carr\'e dans $\Q[t]$, une m\'ethode de
 fibration simple  \cite[\S 3]{CTX2} donne le r\'esultat ci-dessus. 
 
 Par contre, quand on consid\`ere le probl\`eme
 avec 2 variables, la question de l'existence
 et de la densit\'e des solutions enti\`eres d'une \'equation aussi simple que
 $$x^2+ay^2=p(t),$$
 avec $a \in k$ non carr\'e et  $p(t)$ polyn\^ome s\'eparable de degr\'e au moins 3, 
 est hors d'atteinte de toutes les techniques connues. 
 Il  y a deux difficult\'es essentielles : d'une part  pour toute $k$-fibre lisse $X_{t}$   
 le quotient $\Br \, X_{t}/\Br\,  k$ est infini, d'autre part les fibres  correspondant aux z\'eros du polyn\^ome $p(t)$   ne sont pas scind\'ees.

 \medskip
 
 Disons maintenant un mot sur les m\'ethodes employ\'ees dans l'article.
 
 Dans une s\'erie d'articles \cite{Hara,Hara2,Hara3}, le second auteur a \'etudi\'e le principe de Hasse et l'approximation faible
 en familles pour les points rationnels, en tenant compte des contraintes provenant du groupe de Brauer. Ces articles
 supposent la fibration $f : X \to \A^1_{k}$ propre.  Les d\'emonstrations comportent essentiellement deux aspects :
 un aspect alg\'ebrique,
 l'\'etude du comportement  du groupe de Brauer en famille et un aspect  arithm\'etique, le
  choix d'une  fibre dont l'ensemble de Brauer-Manin est non vide.
 
 \medskip 
 
Pour le probl\`eme que nous consid\'erons ici, la fibration n'est pas propre. 
La fibre g\'en\'erique est un espace homog\`ene d'un groupe semisimple simplement connexe.
La bonne connaissance que nous avons du groupe de Picard et du groupe de Brauer
de tels espaces nous permet de les \'etudier en famille, ce qui donne la partie alg\'ebrique
de la d\'emonstration (Th\'eor\`eme \ref{hilbertirr}).

Pour l'aspect arithm\'etique, dans le contexte d'approximation forte
du pr\'esent article, un probl\`eme nouveau se pr\'esente : on doit travailler avec une place exceptionnelle
$v_{0}$ qui est impos\'ee au d\'epart.
Pour le r\'esoudre,  on  d\'eveloppe {\it une variante nouvelle 
de la m\'ethode des fibrations} (m\'ethode qui avait \'et\'e 
employ\'ee dans \cite{Hara}, \cite{Hara2} et \cite{Hara3})~: 
le principal
ingr\'edient suppl\'ementaire est une version raffin\'ee du th\'eor\`eme 
d'approximation forte combin\'ee \`a un argument de type irr\'eductibilit\'e
de Hilbert (proposition~\ref{appforteHilbertraf}).
  
Ces techniques permettent une r\'eduction du probl\`eme d'approximation des points entiers sur  $X$
au cas des points entiers d'une fibre convenable du morphisme 
$f : X \to \A^1_{k} $, fibre qui est un espace homog\`ene 
d'un $k$-groupe semisimple simplement connexe, espace auquel on peut 
 appliquer les r\'esultats sur l'obstruction de Brauer-Manin enti\`ere
 \'etablis par Fei Xu et le premier auteur \cite{CTX}, puis g\'en\'eralis\'es
 par Borovoi et Demarche  \cite{borodemarche}.

\medskip
 
{\bf Conventions et notations}

Soit $k$ un corps. Une $k$-vari\'et\'e $X$ est par d\'efinition un $k$-sch\'ema s\'epar\'e de type fini.
On note $k[X]^{\times}$ le groupe des fonctions inversibles sur $X$.

Si la $k$-vari\'et\'e $X$ est int\`egre, on note $k(X)$ son corps des fonctions rationnelles.

Une $k$-vari\'et\'e est dite {\it scind\'ee} si elle contient 
un ouvert non vide qui comme $k$-vari\'et\'e 
est g\'eom\'etriquement int\`egre.

Soit $\k$ une cl\^oture s\'eparable de $k$. Pour toute $k$-vari\'et\'e $X$, 
on note
$\overline{X}$ la $\k$-vari\'et\'e $X \times_{k}\k$. 
De m\^eme si $S$ est un sch\'ema et 
$\X$ un $S$-sch\'ema, on note $\X_T:=\X \times_S T$ pour tout $S$-sch\'ema 
$T$.
 
  On note $\pic X$ le groupe de Picard d'un sch\'ema $X$ et 
$\Br X=H^2_{\et}(X,\G_{m})$ son groupe de Brauer. 
On d\'esigne aussi par $X^{(1)}$ l'ensemble 
des points de codimension $1$ de $X$.

 Soit $k$ un corps de nombres. Pour $v$ une place de $k$ on note $k_{v}$ le compl\'et\'e
 de $k$ en $v$, et pour $v$ non archim\'edienne,  on note $O_{v} \subset k_{v}$ l'anneau des entiers.
 Pour $S$ un ensemble fini de places de $k$, on note $\calo_S \subset k$ l'anneau des entiers hors de $S$.
 
  Pour la d\'efinition et les g\'en\'eralit\'es sur l'obstruction de Brauer--Manin enti\`ere,
  nous renvoyons le lecteur aux introductions de \cite{CTX} et \cite{CTX2}.
En particulier si $X$ est une $k$-vari\'et\'e, on note $X(\A_k)$ (resp. 
$X(\A_k ^{v_0})$) l'ensemble
de ses points ad\'eliques (resp. de ses points ad\'eliques hors de $v_0$); 
pour tout sous-ensemble $B$ de $\Br X$, on note
$X(\A_k)^B$ le sous-ensemble de $X(\A_k)$ constitu\'e des points ad\'eliques 
orthogonaux \`a $B$ pour l'accouplement de Brauer-Manin et on pose 
$X(\A_k)^{\Br}:=X(\A_k)^{\Br X}$.

\section{Sp\'ecialisation du groupe de Brauer}\label{picbraueresphomogene}

Le but de ce paragraphe est de d\'emontrer le th\'eor\`eme \ref{hilbertirr}, analogue dans
le pr\'esent cadre du th\'eor\`eme 2.3.1 de \cite{Hara2}, qui traitait de fibrations propres
sur la droite affine.
 Il   faut  adapter au pr\'esent contexte tous  les arguments du \S 2 de \cite{Hara2}.

\begin{prop}\label{specialisation}
  Soit $R$ un anneau de valuation discr\`ete de corps r\'esiduel $k$ de caract\'eristique z\'ero.
Soit $K$ le corps des fractions de $R$. Soit $\X$ un $R$-sch\'ema lisse \`a fibres g\'eom\'etriquement int\`egres.
Soit $p : \X \to \spec R$ le morphisme structural.
Soit $X/K$ la fibre g\'en\'erique et $Y/k$ la fibre sp\'eciale.
Supposons $H^1_{\et}({\overline Y},\Q/\Z)=0$, c'est-\`a-dire 
que la fibre sp\'eciale g\'eom\'etrique n'a pas de rev\^etement ab\'elien connexe non trivial.
Alors la fl\`eche de restriction $\Br \X \to \Br X $ induit un isomorphisme
$$ \Br \X/\Br R \oi \Br X /  \Br K$$
et une fl\`eche de sp\'ecialisation 
$$ \Br X / \Br K  \to \Br Y/\Br k.$$
\end{prop}
(On fait syst\'ematiquement l'abus de langage d'\'ecrire $A/B$ plut\^ot 
que $A/{\rm Im} \, (B)$.)
\medskip

\begin{proof}

On a le  diagramme commutatif de suites exactes
{\small
$$\begin{array}{ccccccccccc}

0 &\to&  \Br R &\to& \Br K &\to& H^1(k,\Q/\Z) &  \to & 0 \cr

&& \downarrow && \downarrow& &  \downarrow & &  \\
 
0 &\to&  \Br \X &\to& \Br X &\to& H^1_{\et}(Y,\Q/\Z) & & \cr
 \end{array}
$$}
La premi\`ere suite exacte est bien connue \cite{groth}, \`a la surjectivit\'e \`a droite pr\`es,
pour laquelle nous renvoyons \`a  \cite{AB} et \cite[Thm. B. 2.1]{CTHK}.
 
La deuxi\`eme suite exacte r\'esulte  
de la suite de localisation en cohomologie \'etale, 
de la suite exacte de Kummer, de l'injectivit\'e
du groupe de Brauer d'un sch\'ema r\'egulier int\`egre dans
le groupe de Brauer de son corps des fonctions, et
du th\'eor\`eme de puret\'e de Gabber (voir \cite{fujiwara}).
Sous les hypoth\`eses faites sur $Y$,  on a $H^{0}_{\et}(\ov{Y},\Q/\Z)=\Q/\Z$ et $H^{1}_{\et}(\ov{Y},\Q/\Z)=0$,
donc la fl\`eche naturelle
$$ H^1(k,\Q/\Z) \to  H^1_{\et}(Y,\Q/\Z)$$
est un isomorphisme. La proposition r\'esulte alors
du diagramme.
\end{proof}

\begin{prop}\label{suitespectrale}
Soit $R$ un anneau de valuation discr\`ete.
Soit $K$ le corps des fractions de $R$. Soit $\X$ un $R$-sch\'ema lisse \`a fibres g\'eom\'etriquement int\`egres.
Soit $p : \X \to \spec R$ le morphisme structural.
Soit $X=\X_{K}$ la fibre g\'en\'erique.

(1) La fl\`eche de restriction $\pic \X \to \pic X$ est un isomorphisme.

(2) Soit $\spec S \to \spec R$ un rev\^etement galoisien (\'eventuellement infini) de groupe 
de Galois~$G$.  Supposons que l'on a : 

(a) $S^{\times}  \oi H^0(\X_{S}, \G_{m})$;

(b) $\Br S=0$;

(c) $H^3(G,S^{\times})=0$.

Alors on a une suite exacte naturelle
$$ 0 \to H^1(G,\pic \X_{S}) \to \Br \X/ \Br R \to [\Br \X_{S}]^G \to H^2(G,\pic \X_{S}) .$$

\end{prop}

\begin{proof}
L'\'enonc\'e (1) est classique. La suite exacte dans (2) provient de la suite spectrale
de Leray pour le morphisme $\X \to \spec R$ et le faisceau $\G_{m}$, la d\'emonstration
est essentiellement celle de  \cite[Prop. 2.1]{Hara2}.
\end{proof}

\begin{prop}\label{tresgranddiagramme}
Soit $R$ un anneau de valuation discr\`ete de corps r\'esiduel $k$  de caract\'eristique z\'ero.
Soit $K$ le corps des fractions de $R$.  Soit $R^h$ le hens\'elis\'e de $R$ et $R^{sh}$
le hens\'elis\'e strict.
Soit $K^h$, resp. $K^{sh}$  le corps des fractions de $R^h$, resp. $R^{sh}$. 
Fixons des inclusions $K \subset K^h \subset  K^{sh} \subset \ov{K}$.
Notons $\Ga = {\rm Gal}(\ov{k}/k)={\rm Gal}(K^{sh}/K^h)$ et $\Ga_{K}={\rm Gal}(\ov{K}/K)$.

Supposons $H^3(\Ga,\ov{k}^{\times})=0$, $H^3(\Ga,\Z)=0$ et 
$H^3(\Ga_{K},\ov{K}^{\times})=0$.

 Soit $\X$ un $R$-sch\'ema lisse \`a fibres g\'eom\'etriquement int\`egres.
Soit $p : \X \to \spec R$ le morphisme structural.
Soit $X=\X_{K}$ la fibre g\'en\'erique.

Supposons que l'on a $\G_{m,R} \oi p_{*}\G_{m,\X}$ universellement sur $\spec R$.

On a alors un diagramme commutatif de suites exactes

{\small
$$\begin{array}{ccccccccccc}
0 &\to&  H^1(K,\pic X_{\ov K}) &\to& \Br X/\Br K &\to&
[\Br X_{ \ov K}]^{\Ga_K} &\to&  H^2(K,\pic X_{\ov K}) \cr

&& \downarrow && \downarrow& &  \downarrow && \downarrow &  & \\
0 &\to&  H^1(K^h,\pic X_{\ov K}) &\to&  \Br X_{K^h}/\Br K^h &\to& 
[\Br X_{ \ov K}]^{\Ga_{K^h}} &\to&  H^2(K^h,\pic X_{\ov K}) \cr

&&\uparrow && || & &  \uparrow&&  \uparrow&  & \\
0 &\to&  H^1(\Ga,\pic X_{K^{sh} } ) &\to&  \Br X_{K^{h}}/\Br K^h &\to& 
[\Br X_{ K^{sh}}]^{\Ga}     &\to&  H^2(\Ga,\pic X_{K^{sh}}) \cr

&&||&& \uparrow{\rho}& &  \uparrow &&  ||&  & \\
0 &\to&  H^1(\Gamma,\pic \X_{R^{sh} } ) &\to&  \Br \X_{R^{h}}/\Br R^h &\to& 
[\Br \X_{ R^{sh}}]^{\Ga}     &\to&  H^2(\Ga,\pic \X_{R^{sh}}) \cr

&&\downarrow&& \downarrow& & \downarrow&&  \downarrow&  & \\
0 &\to&  H^1(\Ga,\pic \ov Y) &\to&  \Br Y/\Br k &\to&  
\Br \ov Y^{\Ga} &\to&  H^2(\Ga,\pic \ov Y).
\end{array}
$$}

Si de plus $H^1(\ov{Y},\Q/\Z)=0$, la fl\`eche $\rho$ dans ce diagramme
est un isomorphisme.    
\end{prop}

\begin{proof} 
 V\'erifions que les hypoth\`eses de la proposition \ref{suitespectrale} sont satisfaites
pour chacune des 5 lignes du diagramme. L'hypoth\`ese (a) est par hypoth\`ese
satisfaite pour toutes.

Pour la ligne 5, (b) est automatique et (c) est l'hypoth\`ese $H^3(k,\G_{m})=0$.

Pour la ligne 4, $\Br R^{sh} =\Br \ov{k}=0$. Par ailleurs on a la suite exacte scind\'ee
de $\Ga$-modules galoisiens donn\'ee par la r\'eduction modulo l'id\'eal maximal
$$1 \to {R^{sh}}^1 \to {R^{sh}}^{\times} \to \ov{k}^{\times} \to 1$$
o\`u ${R^{sh}}^1$ est uniquement divisible. L'hypoth\`ese $H^3(k,\G_{m})=0$
assure donc $H^3(\Ga,{R^{sh}}^{\times})=0$.

Pour la ligne 3, on a la suite scind\'ee de $\Ga$-modules 
$$1 \to {R^{sh}}^{\times}  \to {K^{sh}}^{\times} \to \Z \to  0$$
donn\'ee par la valuation.
On a vu que  l'\'egalit\'e $H^3(k,\G_{m})=0$ implique alors 
$H^3(\Ga,{R^{sh}}^{\times})=0$.
Sous l'hypoth\`ese $H^3(\Ga,\Z)=0$ on a donc $H^3(\Ga,{K^{sh}}^{\times})=0$,
c'est-\`a-dire la condition (c).
Comme $\ov{k}$ est alg\'ebriquement clos, on a $$0=\Br R^{sh}=\Br K^{sh}.$$

Pour la ligne 2, l'hypoth\`ese (b) est automatique et 
(c) se lit $H^3(K^h,\G_{m})=0$. D'apr\`es \cite{milne}, 
exemple (c) p. 108, il suffit de voir
que $H^3(R^h,\G_{m})$ et $H^2(k,\Q/\Z)$ sont nuls. Or 
$H^2(k,\Q/\Z)=H^3(\Ga,\Z)$ et   $H^3(R^h,\G_{m})$
 est isomorphe \`a $H^3(k,\G_{m})$ (\cite{milne}, 
Rem. III.3.11).

Pour la ligne 1, l'hypoth\`ese (b) est automatique et (c) se
lit $H^3(\Ga_{K},\ov{K}^{\times})=0$.

La proposition \ref{suitespectrale}  et les hypoth\`eses
sur la cohomologie de $K$ donnent  donc le  diagramme de suites exactes ci-dessus,
 et la proposition \ref{specialisation}. 
montre que, sous l'hypoth\`ese $H^1(\ov{Y},\Q/\Z)=0$,  l'homomorphisme
$\rho :  \Br \X_{R^{h}}/\Br R^h   \to  \Br X_{K^{h}}/\Br K^h$
est un isomorphisme.
 \end{proof}

\begin{rem}\label{surk(t)}
{\rm Les hypoth\`eses  $H^3(\Ga,\ov{k}^{\times})=0$, $H^3(\Ga,\Z)=0$ et 
$H^3(\Ga_{K},\ov{K}^{\times})=0$ sont satisfaites lorsque $k$ est un 
corps de nombres et $K=k(t)$; voir par exemple \cite{cassfro} p. 199
et \cite{Hara}, p. 241.} 
\end{rem}

\begin{prop}\label{stricthenselien}
  Soient $k$ un corps de caract\'eristique z\'ero, $\ov k$ une cl\^oture alg\'ebrique,
et  $\Ga={\rm Gal}({\ov k}/k)$. 
  Soit $X$ une $k$-vari\'et\'e alg\'ebrique
 lisse et g\'eom\'etri\-que\-ment int\`egre, et $f :  X \to \A^1_k$ un $k$-morphisme.
 Soit $K=k(\A^1)$ et soit $G$ un $K$-groupe semisimple simplement connexe.
 Supposons que la fibre g\'en\'erique de $f$ est un espace homog\`ene de $G$, 
\`a stabilisateurs g\'eom\'etriques
  r\'eductifs 
  connexes.
  Il existe un ouvert non vide $U \subset \A^1_k$ tel que
 pour tout $k$-point $M$ de $U$ d'anneau local $R$, avec $R$
de hens\'elis\'e $R^h$ et de hens\'elis\'e strict $R^{sh}$, 
on ait
(notant  $Y/k$ la fibre en $M$ de $f$)~:

(i) Les fl\`eches naturelles 
$$\pic X_{R^{sh}} \to \pic X_{K^{sh}} \to \pic X_{\ov{K}}$$
sont des isomorphismes de r\'eseaux, et la
 fl\`eche naturelle   
$$\pic X_{R^{sh}} \to \pic {\ov Y}$$
est un isomorphisme de $\Ga$-r\'eseaux, o\`u on a not\'e 
$X_{R^{sh}}=X \times_{\A^1_k} R^{sh}$ via la fl\`eche 
$\spec R \to \A^1_k$ associ\'ee \`a $M$; de m\^eme 
$X_{K^{sh}}$ et $X_{\ov{K}}$ sont d\'efinis par changement de base 
\`a partir de la fibre g\'en\'erique $X_K$.

(ii) Les fl\`eches naturelles
$$\Br X_{R^{sh}} \to \Br X_{K^{sh}} \to \Br X_{\ov{K}} $$
sont des isomorphismes de   groupes ab\'eliens finis,
et la fl\`eche de sp\'ecialisation
$$\Br X_{R^{sh}} \to \Br {\ov Y}$$
est un isomorphisme de $\Ga$-modules finis.

(iii) Supposons $H^3(\Ga,\ov{k}^{\times})=0$, 
ce qui est le cas si $k$ est un corps de nombres.
Alors la fl\`eche naturelle 
$$\Br X_{R^h}/\Br R^h \to \Br Y/ \Br k$$
est un isomorphisme de groupes ab\'eliens finis.
\end{prop}

\begin{proof}

 Il existe un ouvert non vide $U \subset \A^1_k$
et un rev\^etement fini \'etale galoisien connexe $V \to U$,
d\'efinissant une extension galoisienne finie $L/K$ de corps,
 tels qu'on ait les propri\'et\'es suivantes~:

(a) Le groupe $G$ s'\'etend en un $U$-groupe semisimple simplement connexe.

(b) Il existe une section $\sigma$ de $f_{V} : X_{V} \to V$.

(c) Le stabilisateur $H$ de cette section est un $V$-groupe r\'eductif   (\`a fibres connexes).

(d) Le groupe $H$ s'inscrit dans une extension de $V$-groupes
$$ 1 \to H^{ss} \to H \to T \to 1$$
o\`u $T$ est un $V$-tore et $H^{ss} $ est un $V$-groupe semisimple, groupe d\'eriv\'e
de $H$.

 (e) On a une suite exacte de $V$-groupes
$$ 1 \to \mu \to H^{sc} \to H^{ss} \to 1,$$
o\`u $H^{sc}$ est un $V$-groupe semisimple simplement connexe,
et $\mu$ est un $V$-sch\'ema en groupes finis ab\'eliens \'etales.

On dispose du $H$-torseur $G_{V} \to  X_{V}$ d\'efini par la section $\sigma$.
Ce torseur a une restriction triviale au-dessus de l'image de $\sigma : V \to X_{V}$.

\`A une telle situation  
 sont associ\'es, de fa\c con fonctorielle en tout $V$-sch\'ema $W$,
des homomorphismes

$$\G_{m}(W) \to \G_{m}(X_{W}),$$
$${\widehat T}(W)  \oi  \widehat{H}(W)  \to \pic X_{W},$$
o\`u la   fl\`eche $\widehat{H}(W)  \to \pic X_{W}$ est associ\'ee  au $H_{W}$-torseur  $G_{W} \to X_{W}$, et
$$ \Ext^c_{W-gp}(H_{W},\G_{m,W}) \to \pic H_{W,}$$
o\`u $\Ext^c$ d\'esigne le groupe des classes d'extensions centrales; 
cette derni\`ere fl\`eche est d\'efinie via le fait qu'une telle extension
centrale d\'efinit ipso facto un $H_{W}$-torseur sous $\G_{m,W}$. 

\begin{lem}  \label{injectlem}
Soit $A$ un anneau int\`egre. Soit $H_{A}$ un $A$-sch\'ema en groupes 
r\'eductifs connexes.
Alors la fl\`eche 
$$\Ext^c_{A-gp}(H_{A},\G_{m,A}) \to \pic H_{A,}$$
d\'efinie comme ci-dessus est injective.
\end{lem} 

\begin{proof}
Soit $F$ le corps des fractions de $A$. Soit 
$$1 \to \G_{m,A} \to E \to H_A \to 1$$
une extension centrale dont l'image dans $\pic H_{A}$ est nulle.
Cela signifie qu'il existe une section $\sigma_A : H_A \to E$ 
du morphisme de $A$-sch\'emas $E \to H_A$, section dont on peut supposer
qu'elle envoie le neutre sur le neutre (quitte \`a la multiplier 
par un \'el\'ement de $\G_{m,A}(A)$). Il s'agit alors de montrer 
que $\sigma_A$ est de plus un homomorphisme 
de $A$-sch\'emas en groupes. L'argument de la 
proposition 3.2 dans \cite{CTresol} (reposant sur le lemme de 
Rosenlicht) donne alors que la restriction de 
$\sigma_A$ \`a la fibre g\'en\'erique est un morphisme 
$H_F \to E_F$ de $F$-sch\'emas en groupes. Ceci implique que le morphisme 
de $A$-sch\'emas s\'epar\'es
$$H_A \times H_A \to E \quad (x,y) \mapsto \sigma_A(xy)\sigma_A(y)^{-1} 
\sigma_A(x)^{-1}$$
est constant \'egal au neutre sur la fibre g\'en\'erique, donc partout;
d'o\`u le r\'esultat.
\end{proof}

Reprenons la preuve de la proposition~\ref{stricthenselien}. 
Pour tout $V$-sch\'ema $W$ on a aussi un homomorphisme
$$ \Ext^c_{W-gp}(H_{W},\G_{m,W})   \to \Br_{e} X_{W},$$
o\`u $ \Br_{e} X_{W} \subset \Br X_{W} $ est le sous-groupe form\'e des \'el\'ements
nuls sur $\sigma(W)$, et
o\`u la   fl\`eche est d\'efinie par cup-produit
avec la classe dans $H^1_{\et}(X_{W}, H_{W})$ du $H_{W}$-torseur  $G_{W} \to X_{W} $, 
 et des homå?omorphismes
 $$\pic H_{W}  \to \pic H^{ss}_{W},$$
 $$\hat{\mu}(W)  \to \pic H^{ss}_{W},$$
 ce  dernier  \'etant associ\'e au $\mu_{W}$-torseur $H^{sc}_{W} \to H^{ss}_{W}$.
 
 Pour $W$ le spectre d'un corps, ces diverses applications sont discut\'ees
 dans \cite{sansuc}, \cite{CTX} et \cite{borodemarche}.
 
 On a donc des homomorphismes
\begin{equation} \label{3etoiles}
 \hat{\mu}(W)  \to \pic H^{ss}_{W}  \leftarrow \pic H_{W}  \leftarrow  \Ext^c_{W-gp}(H_{W},\G_{m,W}) \to  \Br_{e} X_{W}  
\end{equation}

 \medskip
 
 Soit $F$ un corps de caract\'eristique z\'ero.  
 Pour un $F$-groupe $G$ semisimple simplement connexe,  
  on a les propri\'et\'es   suivantes 
  $$ F^{\times} \oi F[G]^{\times};$$
   $$ \pic G=0;$$
 $$ \Br F  \oi \Br G.$$
 Pour $F$ alg\'ebriquement clos, les deux premi\`eres propri\'et\'es sont bien connues \cite{sansuc}. La troisi\`eme l'est
 aussi, elle est \'etablie de fa\c con alg\'ebrique dans \cite{sgille}.
 Pour $F$ quelconque, on en d\'eduit le r\'esultat g\'en\'eral en utilisant la suite
 spectrale de Hochschild-Serre pour la cohomologie \'etale du faisceau $\G_{m}$.

\medskip
 
Supposons que
 $W$ est le spectre d'un corps $F$.
  La proposition 6.10 de  \cite{sansuc} montre que la fl\`eche naturelle
 $$ \widehat{T}(F) \to \pic X_{F}$$
 est un isomorphisme.
 
 Consid\'erons la suite d'homomorphismes :
  $$ \hat{\mu}(F)  \to \pic H^{ss}_{F}  \leftarrow \pic H_{F}  \leftarrow  \Ext^c_{F-gp}(H_{F},\G_{m,F}) \to  \Br_{e} X_{F}.  $$

 La premi\`ere fl\`eche est un isomorphisme  (\cite[Lemme 6.9]{sansuc}).
 La seconde fl\`eche est un isomorphisme si $F$ est alg\'ebriquement clos
  (\cite[Cor. 6.11 et Rem. 6.11.3]{sansuc}).
  La troisi\`eme fl\`eche est un isomorphisme (\cite[Cor. 5.7]{CTresol}).
  La quatri\`eme fl\`eche est un isomorphisme; ceci r\'esulte de
    \cite[Thm. 2.8]{borodemarche} et des  \'egalit\'es $ \pic G_{F}=0$
et $ \Br F  = \Br G_{F}.$

 Soit $M$  un $k$-point de $U$. Soit $R$ le hens\'elis\'e de $\A^1_{k}$
 en $M$ et $R^{sh}$ le hens\'elis\'e strict. Fixons une factorisation
 $ \spec R^{sh} \to V \to U$ induisant $\spec R^{sh} \to \spec R$.
 Fixons aussi des plongements $K \subset K^h \subset K^{sh} \subset {\ov K}$.
 Soit $Y/k$ la fibre de $f :X \to \A^1_{k}$ en $M$.
  
 On a un diagramme commutatif

$$\begin{array}{ccccccccccc}
\widehat{T}(R^{sh})  &\to&  \pic X_{R^{sh}} \cr
\downarrow && \downarrow \cr
\widehat{T}(K^{sh}) &\to&  \pic X_{K^{sh}} \cr
\downarrow && \downarrow \cr
\widehat{T}(\ov{K}) &\to&  \pic X_{\ov{K}}, \cr
 \end{array}
$$
Comme $T_{R^{sh}}$
est un tore d\'eploy\'e, les fl\`eches verticales de gauche sont des
isomorphismes.
Comme on a dit, la proposition 6.10 de \cite{sansuc}
implique que les deux fl\`eches horizontales inf\'erieures
sont des isomorphismes. Le (1) de la proposition  \ref{suitespectrale} 
montre que la verticale sup\'erieure droite est un isomorphisme.
On conclut que toutes les fl\`eches dans ce diagramme sont
des isomorphismes.

On a par ailleurs un diagramme commutatif
$$\begin{array}{ccccccccccc}
\widehat{T}(R^{sh})  &\to&  \pic X_{R^{sh}} \cr
\downarrow && \downarrow \cr
\widehat{T}({\ov k})  &\to&  \pic \ov{Y}
 \end{array}
$$
dans lequel toutes les fl\`eches sauf peut-\^etre la fl\`eche
 $ \pic X_{R^{sh}} \to  \pic \ov{Y}  $
 sont des isomorphismes. Ainsi cette derni\`ere fl\`eche est un isomorphisme,
 ce qui ach\`eve d'\'etablir l'\'enonc\'e  (i).
 
Sur $\ov{K}$, tous les groupes intervenant dans (\ref{3etoiles}) 
sont isomorphes
au groupe fini $\hat{\mu}(\ov{K})$. 
Les groupes $\Br X_{\ov K}$ et
   $\Ext^c_{{\ov K}-gp}(H_{{\ov K}},\G_{m,{\ov K}})$
sont donc finis.

Soit $L/K$ une extension finie galoisienne
de corps telle que l'application  
$\Br X_{L} \to \Br X_{\ov K}$
soit surjective et que l'application 
$$\Ext^c_{{L}-gp}(H_{{L}},\G_{m,{L}}) \to \Ext^c_{{\ov K}-gp}(H_{{\ov K}},\G_{m,{\ov K}})$$
soit surjective.  

Quitte \`a restreindre $U$, on peut de plus  supposer que l'extension $L/K$ est non ramifi\'ee sur $U$,
et que tout \'el\'ement de $\Ext^c_{{\ov K}-gp}(H_{{\ov K}},\G_{m,{\ov K}})$
provient d'un \'el\'ement de $\Ext^c_{S-gp}(H_{S},\G_{m,S})$,
o\`u $S$ est la fermeture int\'egrale de $U$ dans $L$.
 
Pour l'anneau local $R$ d'un point   $M\in U(k)$, on a
$$K \subset L \subset K^{sh} \subset \ov{K},$$
 l'application
 $ \Br X_{K^{sh}} \to \Br   X_{\ov K}$ 
est surjective et 
  l'application 
$$ \Ext^c_{ {R^{sh}}-gp}(H_{ {R^{sh}}},\G_{m, {R^{sh}}})  \to \Ext^c_{{\ov K}-gp}(H_{{\ov K}},\G_{m,{\ov K}})$$
est surjective.

\medskip

D'apr\`es  (\ref{3etoiles}),
on a un diagramme commutatif d'applications naturelles $$\begin{array}{ccccccccccc}
  \hat{\mu}({\ov K}) & \to & \pic H^{ss}_{{\ov K}}  &\leftarrow & \pic H_{{\ov K}} & \leftarrow & \Ext^c_{{\ov K}-gp}(H_{{\ov K}},\G_{m,{\ov K}})& \to&  \Br_{e} X_{{\ov K}}\cr
\uparrow && \uparrow&&\uparrow&&\uparrow && \uparrow  \cr
\hat{\mu}({K^{sh}}) & \to&  \pic H^{ss}_{{K^{sh}}}   &\leftarrow &  \pic H_{{K^{sh}}}   & \leftarrow & 
   \Ext^c_{K^{sh}-gp}(H_{K^{sh}},\G_{m,K^{sh}})  
  & \to &  \Br_{e} X_{  K^{sh}  }  \cr
\uparrow && \uparrow&&\uparrow&&\uparrow && \uparrow \cr
  \hat{\mu}( {R^{sh}})  & \to&  \pic H^{ss}_{ {R^{sh}}}   &\leftarrow & \pic H_{ {R^{sh}}}  &\leftarrow & \Ext^c_{ {R^{sh}}-gp}(H_{ {R^{sh}}},\G_{m, {R^{sh}}}) & \to&   \Br_{e} X_{ {R^{sh}} }\cr
\downarrow && \downarrow&&\downarrow&&\downarrow && \downarrow \cr
 \hat{\mu}({\ov k})  & \to&  \pic H^{ss}_{{\ov k}}  &\leftarrow & \pic H_{{\ov k}}   &\leftarrow & \Ext^c_{{\ov k}-gp}(H_{{\ov k}},\G_{m,{\ov k}}) & \to&   \Br_{e} \ov{Y} 
 \end{array}
$$

Les fl\`eches dans la verticale de gauche sont toutes des isomorphismes.
Dans la premi\`ere ligne, toutes les fl\`eches sont des isomorphismes.
Dans la seconde ligne, toutes les fl\`eches sauf peut-\^etre la fl\`eche
$ \pic H^{ss}_{{K^{sh}}}   \leftarrow   \pic H_{{K^{sh}}} $ sont des isomorphismes.
La fl\`eche $  \pic H^{ss}_{{K^{sh}}}   \leftarrow   \pic H_{{K^{sh}}} $ est injective via la proposition~6.10 de \cite{sansuc}, 
car $\pic T_{K^{sh}}=0$.
Comme la fl\`eche $\Br_{e} X_{  K^{sh}  } \to  \Br_{e} X_{{\ov K}}$ est par nos hypoth\`eses
surjective, on conclut que toutes les fl\`eches dans les deux premi\`eres lignes,
ainsi qu'entre ces deux lignes, sont des isomorphismes.

Les fl\`eches verticales entre la deuxi\`eme et la troisi\`eme ligne, sauf peut-\^etre les deux derni\`eres
 \`a droite, sont clairement des isomorphismes. On en d\'eduit que les deux premi\`eres fl\`eches de la troisi\`eme ligne
 sont des isomorphismes.
 La derni\`ere fl\`eche verticale entre les lignes trois et deux est une injection (r\'egularit\'e des
 sch\'emas consid\'er\'es).  
 
 Comme $Y$ est un espace homog\`ene d'un groupe semisimple simplement connexe
 pour une action dont les stabilisateurs sont connexes, on a  $H^1(\ov{Y},\Q/\Z)=0$ et la proposition
  \ref{specialisation} montre que  la fl\`eche en question est un isomorphisme.

 La derni\`ere ligne est compos\'ee d'isomorphismes.
On en conclut que les trois premi\`eres fl\`eches verticales entre la troisi\`eme et la
quatri\`eme ligne sont des isomorphismes.

D'apr\`es ce qui pr\'ec\`ede, l'application  
$$ \Ext^c_{ {R^{sh}}-gp}(H_{ R^{sh}} ,\G_{m, R^{sh} })
 \to  
  \Ext^c_{K^{sh}-gp}(H_{K^{sh}},\G_{m,K^{sh}}) $$
est surjective.

Par ailleurs la fl\`eche  $ \Ext^c_{ {R^{sh}}-gp}(H_{ {R^{sh}}},\G_{m, {R^{sh}}}) \to  \pic H_{ {R^{sh}}}   $  est injective d'apr\`es le 
lemme~\ref{injectlem}.
Une chasse au diagramme donne alors que toutes les fl\`eches sont des isomorphismes (de groupes finis).

Comme   $\Br_{e}(\bullet)=\Br(\bullet)$ pour chacun des groupes dans la verticale de droite,
on a maintenant \'etabli l'\'enonc\'e  (ii).

 L'\'enonc\'e (iii) r\'esulte
alors de la proposition \ref{tresgranddiagramme}.
 et de sa d\'emonstration, qui pour les deux   lignes inf\'erieures du grand diagramme,
 utilise seulement l'hypoth\`ese $H^3(k,\G_{m})=0$. 
\end{proof}

On peut maintenant \'etablir le th\'eor\`eme suivant, qu'on comparera
avec \cite[Thm. 2.3.1]{Hara2}.

   \begin{theo}\label{hilbertirr}
  Soit $k$ un corps de nombres. Soit $X$ une $k$-vari\'et\'e 
 lisse et g\'eom\'etri\-que\-ment int\`egre, et $f :  X \to \A^1_k$ un $k$-morphisme.
 Soit $K=k(\A^1)$ et soit $G$ un $K$-groupe semisimple simplement connexe.
 Supposons que la fibre g\'en\'erique de $f$ est un espace homog\`ene de $G$, 
\`a stabilisateurs g\'eom\'etriques  r\'eductifs  connexes. 
Il existe alors un sous-ensemble hilbertien $Hil$ de $\A^1(k)$
tel que pour tout point $m \in Hil$ la  fibre $X_m$ soit lisse et que 
l'on ait un isomorphisme naturel de groupes finis
$$\Br X_{\eta} /\Br K \oi \Br X_m /\Br k.$$
   \end{theo}

 \begin{proof}
 On reprend les notations de la d\'emonstration pr\'ec\'edente. On
note d\'ej\`a que 
$\G_{m}=p_{*}\G_{m} $ universellement sur la fibre g\'en\'erique 
$p : X_{\eta} \to \spec K$, via le lemme de Rosenlicht 
(car $G$ est semi-simple).

 On suppose que l'extension finie galoisienne
 $L/K$ est suffisamment grosse pour que comme pr\'ec\'edemment
 l'application  $\Br X_{L} \to \Br X_{\ov K}$
soit surjective, mais aussi 
  que l'application  de restriction
 $\pic X_{L}   \to \pic X_{\ov K}$
 soit surjective, et donc un isomorphisme de r\'eseaux, l'injectivit\'e r\'esultant
 de $\ov{K}^{\times} \oi  \ov{K}[X_{\eta}]^{\times}$.
 Soit $\Delta={\rm Gal}(L/K)$.
 
 D'apr\`es le th\'eor\`eme d'irr\'eductibilit\'e de Hilbert, il existe un
 sous-ensemble hilbertien $Hil$ de $\A^1(k)$
tel que pour tout point $m \in Hil$ la  fibre $Y=X_m$ soit lisse,
l'extension $L/K$ soit non ramifi\'ee en $m$ 
et $m$ soit  {\it inerte} dans $L$. En tout tel point on a  donc une extension induite
de corps de nombres $l/k$ de groupe de Galois $\Delta$.
La hens\'elisation donne \'egalement une extension galoisienne $L^h/K^h$ de groupe $\Delta$.
On a les inclusions de corps $$K \subset L \subset L^h  
\subset K^{sh} \subset {\ov K}.$$
Les applications
  $\pic X_{L}  \to \pic X_{L^{h}}  \to \pic X_{K^{sh}} \to  \pic X_{\ov K}$
  sont des isomorphismes de r\'eseaux.

  On consid\`ere le diagramme commutatif de suites exactes 
(proposition~\ref{tresgranddiagramme} et remarque~\ref{surk(t)})~:
{\small
$$\begin{array}{ccccccccccc}
0 &\to&  H^1(K,\pic  X_{\ov K}) &\to& \Br X/\Br K &\to&
[\Br X_{ \ov K}]^{\Ga_K} &\to&  H^2(K,\pic X_{\ov K}) \cr

&& \downarrow && \downarrow& &  \downarrow && \downarrow &  & \\
0 &\to&  H^1(K^h,\pic X_{\ov K}) &\to&  \Br X_{K^h}/\Br K^h &\to& 
[\Br X_{ \ov K}]^{\Ga_{K^h}} &\to&  H^2(K^h,\pic X_{\ov K}) \cr
\end{array}
$$}

On a 
$$ H^1(\Delta,\pic X_{L})= H^1(K, \pic X_{\ov K})$$
et
$ H^1(\Delta,\pic X_{L^h})= H^1(K^h, \pic X_{\ov K}).$
Ainsi la premi\`ere fl\`eche verticale est un isomorphisme.

Comme $\Br X_{L} \to \Br X_{ \ov K}$ est surjectif,
 l'action de $\Ga_{K}$ sur $\Br X_{ \ov K}$ se factorise par $\Delta$.
 On a alors $[\Br X_{ \ov K}]^{\Ga_K} \oi [\Br X_{ \ov K}]^{\Ga_{K^h}}$.
 
 On a, via la suite de restriction-inflation, le diagramme commutatif de suites exactes~: 
 
{\small
$$\begin{array}{ccccccccccc}
 0 &\to&   H^2(\Delta, \pic X_{L}) &\to&  H^2(K,\pic X_{\ov K}) &\to&  H^2(L,\pic X_{\ov K} ) \\
 &&\downarrow &&\downarrow&& \downarrow \\
 0 &\to& H^2(\Delta, \pic X_{L^h}) &\to& H^2(K^h,\pic X_{\ov K} ) &\to&  H^2(L^h,\pic X_{\ov K} ).
  \end{array}
$$}  
  La fl\`eche $H^2(\Delta, \pic X_{L}) \to H^2(\Delta, \pic X_{L^h}) $
  est un isomorphisme.
  
  Comme le groupe $\Br X_{ \ov K}$ est fini, on peut \`a l'avance choisir $L$
  de fa\c con que de plus l'image de $[\Br X_{ \ov K}]^{\Ga_K}$ dans $ H^2(K,\pic X_{\ov K}) $
  s'annule dans $H^2(L,\pic X_{\ov K} )$ et donc dans $H^2(L^h,\pic X_{\ov K} ).$
 On a alors le diagramme commutatif de suites exactes
  {\small
$$\begin{array}{ccccccccccc}
0 &\to&  H^1(K,\pic X_{\ov K}) &\to& \Br X/\Br K &\to&
[\Br X_{ \ov K}]^{\Ga_K} &\to&  H^2(\Delta, \pic X_{L}) \cr
&& \downarrow && \downarrow& &  \downarrow && \downarrow &  & \\
0 &\to&  H^1(K^h,\pic X_{\ov K}) &\to&  \Br X_{K^h}/\Br K^h &\to& 
[\Br X_{ \ov K}]^{\Ga_{K^h}} &\to&  H^2(\Delta, \pic X_{L^h}) \cr
\end{array}
$$}
 dans lequel toutes les fl\`eches sauf peut-\^etre $ \Br X/\Br K \to 
 \Br X_{K^h}/\Br K^h $ sont des isomorphismes. On en conclut que
 cette derni\`ere fl\`eche est un isomorphisme. En combinant avec 
 la proposition \ref{stricthenselien},
  on obtient le th\'eor\`eme.
   \end{proof}
 
 \begin{rem} {\rm Lorsque  les stabilisateurs g\'eom\'etriques  
pour l'action de $G$ sur la fibre
 g\'en\'erique $X_{K}$ sont des tores alg\'ebriques, ce qui sera le cas
 au \S 4,  les d\'emonstrations de ce paragraphe se simplifient.
 On a  en effet alors $\hat{\mu}=0$,  $\Br X_{\ov{K}}=0$, $\Br X_{R^{sh}} =0$, $\Br \ov{Y} =0$.}
 \end{rem}

\section{Le th\'eor\`eme}

On commence par quelques lemmes pr\'eliminaires.

\begin{lem}\label{florence}
 Soient $k$ un corps de nombres, $G$ un $k(t)$-groupe lin\'eaire connexe, $Y$
 une $k(t)$-vari\'et\'e espace homog\`ene de $G$. Soit $v$ une place   de $k$.
 Soit $k_{v}$ le compl\'et\'e de $k$ en $v$ et $k_{v}^{h} \subset k_{v}$ le sous-corps des \'el\'ements
 alg\'ebriques sur $k$.
 Si $Y$ poss\`ede un $k_{v}(t)$-point, alors $Y$ poss\`ede un $k_{v}^{h}(t)$-point.
  \end{lem}
  \begin{proof} 

  Soit $Z \subset \P^N_{k(t)}$ une compactification projective lisse de la $k(t)$-vari\'et\'e quasi-projective $Y$ (une telle compactification 
existe via le th\'eor\`eme de r\'esolution des singularit\'es d'Hironaka).
  Soit ${\mathcal Z} \subset \P^N_{k}\times_{k} \P^1_{k}$ l'adh\'erence sch\'ematique de $Z$.
  L'espace des $k$-morphismes ${\rm Mor}(\P^1_{k}, \P^n_{k})$ est la r\'eunion disjointe
  des $k$-vari\'et\'es ${\rm Mor}_{d}(\P^1_{k}, \P^n_{k})$ param\'etrisant les morphismes de degr\'e $d$
  pour $d$ entier, $d\geq 0$. Ainsi
  l'espace des sections de ${\mathcal Z} \to \P^1_{k}$ est une union disjointe de
 $k$-vari\'et\'es   $W_{i}$. Par hypoth\`ese, il existe une $k_{v}(t)$-section de $Y/k(t)$,
  donc de $Z/k(t)$,
  donc une $k_{v}$-section de ${\mathcal Z} \to \P^1_{k}$. Ainsi l'un des $W_{i}(k_{v})$ est non vide.
  Il est connu (cf. \cite[Chap. 3.6, Cor. 10]{BLR}) que ceci implique
  $W_{i}(k_{v}^h) \neq \emptyset$.
  Par cons\'equent, la $k(t)$-vari\'et\'e lisse $Z$ poss\`ede un $k_{v}^h(t)$-point.
  Un th\'eor\`eme de M. Florence \cite{florence} assure alors que l'espace homog\`ene $Y/k(t)$ poss\`ede un
$k_{v}^h(t)$-point.
 \end{proof}

Dans le lemme suivant, on ne peut pas remplacer $k^h_{v }(t)$ par $k_{v}(t)$.

\begin{lem}\label{surjhenselkt}
Soit $k$ un corps de nombres. Soit $k^h_{v }$
le sous-corps des \'el\'ements  de $k_{v}$ alg\'ebriques sur $k$
(pour $v$ non archim\'edienne c'est le hens\'elis\'e de $k$ en $v$).
Soit $t$ une variable.
L'application naturelle $\Br k(t) \to \Br k^h_{v}(t)$ est surjective.
\end{lem}

  \begin{proof}
  Notons $E=k_{v}^h$. Les suites exactes de Faddeev 
(\cite{gillesza}, Cor. 6.4.6.)
  pour le groupe de Brauer de 
  $k(t)$ et le groupe de Brauer de $E(t)$ sont compatibles. On a donc le diagramme commutatif 
  de suites exactes

  {\small
\begin{equation} \label{henselbrauer}
\begin{CD} 
0 @>>> \Br k @>>> \Br  k(t)     @>>> \bigoplus_{x \in {\A^1_{k}}^{(1)}} H^1(k(x),\Q/\Z) @>>> 0 
\cr
&& @VVV @VVV @VVV   \cr
0 @>>> \Br E @>>> \Br  E(t)     @>>> \bigoplus_{{x \in \A^1_{k}}^{(1)}} [\oplus_{y \in {\A^1_{E}}^{(1)},  \hskip1mm y\mapsto x} H^1(E(y),\Q/\Z)] @>>> 0. 
 \end{CD}
\end{equation}
}

Pour $L$ un corps de nombres et $S$ un ensemble fini de places de $L$,
l'application de restriction 
$$H^1(L,\Q/\Z) \to \prod_{w \in S_L} H^1(L_w,\Q/\Z)$$
est surjective. Ceci est une cons\'equence du th\'eor\`eme de Grunwald-Wang
\cite[Chap.~10, Thm.~5, p.~103]{AT}.
Il en est donc de m\^eme de l'application de restriction
$$H^1(L,\Q/\Z) \to \prod_{w \in S_L} H^1(L_w^h,\Q/\Z).$$
En appliquant ceci \`a chaque $k(x)$ et \`a $S$ l'ensemble des places
de $k(x)$ au-dessus de la place $v$ de $k$, on voit que la fl\`eche verticale
de droite dans le diagramme ci-dessus est surjective.

Pour $L$ un corps de nombres et $w$ une place de $L$, 
les applications naturelles de restriction $H^1(L_{w}^h,\Q/\Z) \to H^1(L_{w},\Q/\Z)$
et $\Br L_{w}^h \to \Br L_{w}$ sont des isomorphismes.
La th\'eorie du corps de classes montre que l'application  $\Br L \to \Br L_{w} $ est  surjective.
Il en est donc de m\^eme de
l'application de restriction $\Br L \to \Br L_{w}^h$.
La fl\`eche verticale de gauche dans le diagramme commutatif ci-dessus
est donc surjective.

Il en r\'esulte que la fl\`eche verticale m\'ediane est surjective.
  \end{proof}

Pour traiter les probl\`emes li\'es \`a la place $v_0$,
nous aurons besoin d'une version fine du th\'eor\`eme d'approximation 
forte, qui est la proposition \ref{appforteHilbertraf} ci-dessous.

\smallskip

Nous utiliserons une version du th\'eor\`eme d'irr\'eductibilit\'e de Hilbert
\'etablie  par Serre \cite[p. 135/136]{serre}
comme une cons\'equence du th\'eor\`eme de Siegel sur la
 finitude des points entiers des courbes.

 \begin{prop} \label{hilbsiegel}
 Soit $k$ un corps de nombres. Soit $\alpha \in k$ et $T$ un ensemble fini de places de $k$.
 Soit $k_{\alpha,T}$ l'ensemble des $t \in k$ tels que $t-\alpha$ soit une unit\'e en dehors de $T$.
 Soit $\Omega$ un ensemble mince dans $k$. Pour $\alpha$ en dehors d'un ensemble fini
 (d\'ependant de $\Omega$),  l'ensemble $\Omega \cap k_{\alpha,T}$ est fini.
\end{prop}
 
Serre \'etablit cette proposition
dans le cas $k=\Q$, et laisse l'\'enonc\'e sur un corps de nombres quelconque
 au lecteur (op. cit., Exercice 1 p. 136). Pour la d\'efinition et les 
propri\'et\'es des ensembles minces, voir le paragraphe~9.1. de 
\cite{serre}.

\begin{prop}\label{appforteHilbertraf}
Soient $k$ un corps de nombres, $k_{i}/k$, $i=1,\dots,n$ des extensions finies,
et pour chaque $i$, $e_{i } \in k_{i}$.
 Soit $N>0$ un entier.

On se donne un ensemble fini $S$ de places de $k$ et une place $v_{0}$ 
hors de $S$.
On suppose qu'il existe au moins une place non archim\'edienne dans $S \cup \{v_{0}\}$.

On se donne des $\lambda_{v} \in k_{v}$ pour $v \in S$, et des  
polyn\^omes irr\'eductibles $P_{j}(z,t) \in k[z,t]$.

Alors il existe  $\theta \in k$
 arbitrairement proche de chaque $\lambda_{v}$ pour $v \in S$,
 avec les propri\'et\'es suivantes :

-- $\theta $  est entier hors de $S \cup \{v_{0}\}$;

--   chaque   $\theta - e_{i}$ est une puissance $N$-i\`eme dans
tout compl\'et\'e de $k_{i}$ en une place   au-dessus de $v_{0}$;

-- chaque $P_{j}(z,\theta)$ est irr\'eductible dans $k[t]$.
\end{prop}

\begin{proof}
Soit $\Omega$ l'ensemble mince form\'e des $\theta \in k$
pour lesquels l'un des $P_{j}(z,\theta) $ est r\'eductible.
Notons $S_{fini}$ l'ensemble des places finies de $S$.

On commence par appliquer l'approximation forte
usuelle \`a l'ensemble des $\lambda_{v}$ pour $v \in S$, ce
qui donne un $\theta_{1} \in k$, entier hors de $S$ et $v_{0}$,
et qui est 
tr\`es proche de $\lambda_{v}$ pour $v \in S$.

Soit $T = S_{fini} \cup \{v_{0}\}$. D'apr\`es la proposition
\ref{hilbsiegel}
on peut de plus  choisir $\theta_{1}$  tel que $\Omega \cap k_{\theta_{1}, T}$
est fini.

Si $v_{0}$ est une place archim\'edienne et $k$ contient $s-1$ autres 
places archim\'ediennes $v_1,...,v_{s-1}$ (avec $s >1$),
on note $u$ une unit\'e de l'anneau des entiers de $k$ de 
valeur absolue strictement plus petite que $1$ en les places
archim\'ediennes autres que $v_{0}$; l'existence d'une telle unit\'e 
est assur\'ee par le fait que l'image des unit\'es de $\calo_k$ dans $\RR^s$
via l'application 
$$x \mapsto (\log (\mid x \mid_{v_0}), ...,\log (\mid x \mid_{v_{s-1}}))$$
est un r\'eseau $L$ tel que le $\RR$-espace vectoriel 
engendr\'e par $L$ est l'hyperplan $\sum_{i=0} ^{s-1} x_i=0$ dans $\RR^s$,
via le th\'eor\`eme des unit\'es de Dirichlet 
(cf. \cite{cassfro}, expos\'e II, paragraphe 18). 
On peut de plus
choisir $u$ positive en $v_{0}$
 si $v_{0}$ est une place r\'eelle.
On cherche
$\theta$ de la forme
$$\theta = \theta_{1} + u^{Ns}.( \prod_{l \in S_{fini}} l)^{Nr},$$
o\`u $l$ est le premier de  $\Q$ induit par $v \in S_{fini}$,
avec   $r>0$ suffisamment grand, {\it puis} $s>0$ suffisamment grand,
 pour que les approximations des $\lambda_{v}$ pour $v \in S$
soient respect\'ees, et que pour tout $i$,  $\theta-e_{i}$ soit positif aux compl\'et\'es
 r\'eels de $k_{i}$ au-dessus de la place $v_{0}$ (noter que comme $s >1$,
on a $\mid u \mid_{v_0} >1$ par la formule du produit).
 Il existe une infinit\'e de tels $\theta$.

\smallskip

Si $v_0$ est l'unique place archim\'edienne de $k$, alors $S_{fini}$ 
est non vide par hypoth\`ese, et on prend $\theta$ de la forme 
$$\theta = \theta_{1} + ( \prod_{l \in S_{fini}} l)^{Nr},$$
avec $r >0$ suffisamment grand comme ci-dessus, ce qui donne encore 
une infinit\'e de $\theta$.

 \medskip

Si enfin $v_{0}$ est une place non archim\'edienne de $k$,  
induisant le premier $p$ sur $\Q$,
on cherche   $\theta$ de la forme
$$\theta = \theta_{1} + ( \prod_{l \in S_{fini}} l)^{Nr} / p^{Ns}$$
avec
 $r>0$ suffisamment grand pour que les approximations des $\lambda_{v}$ pour $v \in S_{fini}$
soient respect\'ees, {\it puis} $s>0 $ suffisamment grand pour que :

-- l'approximation des $\lambda_{v}$ pour  les places $v$ archim\'ediennes soit  respect\'ee;

--  pour chaque $i$ la valuation $p$-adique de
$$( \theta_{1}-e_{i}) p^{Ns} $$ soit suffisamment grande pour que
$$(\theta_{1}-e_{i}) p^{Ns}  +    (\prod_{l \in S_{fini}} l)^{Nr} $$
soit une puissance  $N$-i\`eme dans tout compl\'et\'e   de $k_{i}$ en une place  
au-dessus de la place $v_{0}$, ce qui entra\^{\i}ne que
  $$\theta - e_{i}= \theta_{1} - e_{i}  + ( \prod_{l \in S_{fini}} l)^{Nr} / p^{Ns}
$$
l'est aussi.
Il existe une infinit\'e de tels $\theta$.

\medskip

 On choisit alors $\theta$  qui  ne soit pas dans l'ensemble fini $\Omega \cap k_{\theta_{1},T}$, donc pas dans $\Omega$,
si bien que chaque $P_{j}(z,\theta)$ est irr\'eductible dans $\Q[t]$.

\end{proof}

Sur un corps $k$ quelconque, un $k$-groupe r\'eductif $G$  est dit isotrope s'il
existe un $k$-homomorphisme non constant $\G_{m,k} \to  G$.
Sur un corps local $k$, un $k$-groupe r\'eductif $G$ est isotrope
si et seulement si l'espace topologique $G(k)$ est non compact
(crit\`ere de Godement, Bruhat et Tits, cf. Prasad  \cite{prasad}).

\medskip

Nous sommes maintenant \`a pied d'oeuvre pour \'etablir le th\'eor\`eme principal
de cet article.

 \begin{theo} \label{thprincipal}
 Soit $k$ un corps de nombres. Soient $X$ une $k$-vari\'et\'e 
 lisse et g\'eom\'etri\-que\-ment int\`egre
  et $f :  X \to \A^1_k$ un $k$-morphisme.
 Supposons $X(\A_{k}) \neq \emptyset$.
 Soit $K=k(\A^1)$ et soit $G$ un $K$-groupe semisimple simplement connexe,
 absolument presque $K$-simple.
  
  On suppose :

(i) La fibre g\'en\'erique $X_{\eta}/K$ de $f$ est un espace homog\`ene de $G$ \`a
stabilisateurs r\'eductifs connexes.
 
(ii) Toutes les  fibres de $f$ sont scind\'ees.

(iii) Il existe une place $v_0$ de $k$ telle que~: la fibre g\'en\'erique 
$X_{\eta}$ poss\`ede un $k_{v_0}(t)$-point et il existe un ouvert de 
Zariski $V$ de $ \A^1 _k$ tel que
 la sp\'ecialisation $G_{x}/k_{v_{0}}$ 
du $k(t)$-groupe $G$ en tout point  
 $x \in  V(k_{v_0})$ soit d\'efinie 
 et soit semi-simple, simplement connexe, isotrope.
 
(iv) Pour tout \'el\'ement   $\alpha \in \Br X$,  l'application d'\'evaluation
${\rm ev}_{\alpha} :  X(k_{v_{0}}) \to \Br k_{v_{0}}$ est constante.

\smallskip

Alors l'image diagonale de $X(k)$ dans la projection de $X(\A_k)^{\Br X}$
sur $X(\A_k^{v_0})$ est dense.

 \end{theo}

\begin{proof}
Commen\c cons par \'etablir le lemme suivant.
\begin{lem} Sous les hypoth\`eses du th\'eor\`eme, les groupes  $\Br X_{\eta}/\Br K$  et  $\Br X/\Br k$ sont finis.
\end{lem}
\begin{proof}
D'apr\`es la proposition  \ref{stricthenselien},  sous l'hypoth\`ese (i), d'une part  le module galoisien $\Pic X_{\ov K}$ 
est un r\'eseau,
donc $H^1(K,\Pic X_{\ov K})$ est fini, d'autre part
  le groupe $\Br X_{\ov K}$ est fini. D'apr\`es la proposition  \ref{tresgranddiagramme} et  la remarque  \ref{surk(t)},
  ceci implique que
le groupe $\Br X_{\eta}/\Br K$ est fini. Soit $\alpha \in \Br K$ dont l'image $\beta$
dans $\Br X_{\eta}$ appartient au sous-groupe
$\Br X$.  Les r\'esidus de $\beta$ aux points de codimension 1 de $X$ sont nuls.
L'hypoth\`ese (ii) assure alors que les r\'esidus de $\alpha$ aux points de codimension 1
de $\A^1_{k}$ sont nuls. Ceci implique que $\alpha$ appartient \`a l'image de $\Br k$ dans $\Br K$.
Comme $\Br X_{\eta}/\Br K$ est fini, on en d\'eduit que $\Br X/\Br k$ est fini.
\end{proof}

Il existe  un ensemble fini $S$
de places de $k$
contenant $v_0$ et les places archim\'ediennes et un $\calo_S$-sch\'ema 
lisse de type fini $\X$ \'equip\'e d'un $\calo_S$-morphisme $\X \to \A^1_{\calo_S}$
qui \'etend $f : X \to \A^1_{k}$.

Pour chaque place $v \in S \setminus \{v_{0}\}$, on se donne un ouvert $W_{v} \subset X(k_{v})$.
On suppose que l'ensemble
$$E:=[X(k_{v_{0}})  
\times \prod_{v \in S \setminus \{v_{0}  \} }   W_{v} \times \prod_{v \notin S} \X(O_{v})]^{\Br X}$$
est non vide. 
On veut montrer que l'ensemble
$$X(k_{v_{0}})  
\times \prod_{v \in S \setminus \{v_{0}  \} }   W_{v} \times \prod_{v \notin S} \X(O_{v})$$
 contient un point de $X(k)$.

Notons que chaque $\X(O_{v}) \subset X(k_{v})$ est un ouvert non vide pour 
la topologie $v$-adique.  
Pour \'etablir l'\'enonc\'e ci-dessus, on peut donc remplacer $S$ par tout ensemble fini de places le contenant.

\smallskip

Choisissons un sous-ensemble fini   
$\Ga$ de $\Br X_{\eta}$ dont l'image modulo $\Br K$ est 
le groupe fini $\Br X_{\eta}/\Br K$. 

Par hypoth\`ese, on dispose ici 
d'un $k_{v_0}(t)$-point de la $K$-vari\'et\'e $X_{\eta}$,
et donc (Lemme \ref{florence}) d'un   $k_{v_0}^h(t)$-point $s_{v_{0}}(t)$ de la $K$-vari\'et\'e $X_{\eta}$. 
D'apr\`es le 
lemme~\ref{surjhenselkt}, quitte \`a modifier chaque \'el\'ement de $\Ga$ par 
un \'el\'ement de $\Br K$, on peut imposer
$$\alpha(s_{v_0}(t))=0 \in \Br k_{v_0}(t)$$
pour tout $\alpha \in \Ga$. Ceci vaut encore si on remplace 
$\Ga \subset \Br X_{\eta}$ par le sous-groupe qu'il engendre, 
sous-groupe qui est fini puisque $\Br X_{\eta}$ est de torsion. On note 
ce sous-groupe d\'esormais $B$. 

\smallskip

Une fois fix\'e le groupe fini $B$, il existe un ouvert non vide 
$U_{0} \subset \A^1_{k}$
tel que les \'el\'ements de $B$ soient non ramifi\'es sur $U:=f^{-1}(U_{0})$.
En d'autres termes, on  a $B \subset \Br U$. 
Quitte \`a restreindre $U_{0}$, on peut supposer que $f : U \to U_{0}$
est lisse \`a  fibres g\'eom\'etriquement int\`egres, que le groupe $G/K$ s'\'etend
en un $U_{0}$-sch\'ema en groupes semisimples $G_0$, que $U \to U_{0}$
est un $G_{0}$-espace homog\`ene, et que de plus $s_{v_{0}}(t)$ d\'efinit une 
section de la restriction de $f$ au-dessus de $U_{0} \times_k k^h_{v_{0}}$. 

\smallskip

Notons $\{m_{1}, \dots, m_{l} \} \in \A^1 _k$ les points ferm\'es dans le compl\'ementaire
de $U_{0}$.

Pour tout tel point $m_{i}$, on note $k_{i}$ le corps r\'esiduel en $m_{i}$.
Pour tout $i$, d'apr\`es l'hypoth\`ese (ii), il existe au moins
une composante g\'eom\'etriquement int\`egre de multiplicit\'e 1 
de la fibre $X_{m_{i}}$.
 On fixe  un ouvert lisse non vide $W_{i}$ d'une telle composante.
On note $K_{i}=k_{i}(W_{i})$.
Soit $K'_{i}/K_{i}$ une extension ab\'elienne finie telle que
tous les r\'esidus des \'el\'ements de $B$ au point g\'en\'erique 
de $W_{i}$ appartiennent \`a $H^1(K'_{i}/K_{i}, \Q/\Z)$.
Quitte \`a restreindre $W_{i}$, on peut supposer que la fermeture
int\'egrale $W'_{i}$ de $W_{i}$ dans $K'_{i}$ est finie \'etale sur $W_{i}$.
Soit $k'_{i}$ la fermeture int\'egrale de $k_{i}$ dans $K'_{i}$.

On agrandit $S$ de fa\c con \`a  avoir les propri\'et\'es suivantes :

(a) Le $k$-morphisme $f : U \to U_{0}$ s'\'etend en un $\calo_S$-morphisme lisse
$\U \to \U_{0}$ de $\calo_S$-sch\'emas, avec $\U \subset \X_{\calo_S}$,  tel que pour tout point 
ferm\'e  $x \in \U_{0}$, de corps r\'esiduel le corps fini $F(x)$,  la fibre $f^{-1}(x)$ poss\`ede un $F(x)$-point. 
Comme les fibres de $U \to U_{0}$ sont g\'eom\'etriquement int\`egres,
ceci est possible par les estim\'ees de Lang--Weil (cf. \cite{skoro},
th\'eor\`eme 1, \'etape 3). On peut aussi supposer $\U_{0} \to \Spec \calo_{S}$ surjectif,
ce qui assure $\U(O_{v}) \neq \emptyset$ pour $v \notin S$.

(b) Les polyn\^omes unitaires $P_{i}(t)\in k[t]$ d\'efinissant les points $m_{i}$
sont $\calo_S$-entiers, et le produit $\prod_{i}P_{i}(t)$ a une r\'eduction s\'eparable
en toute place finie  $v \notin S$. Notons ${\bf m}_{i}=\Spec \calo_S[t]/P_{i}(t)$.

(c) Les \'el\'ements de $B$ appartiennent \`a $\Br \,  \U$.

(d) Les morphismes $W_{i} \subset X_{m_{i}} \subset X$, o\`u la premi\`ere fl\`eche
est une immersion ouverte de $k_{i}$-vari\'et\'es et la seconde une immersion ferm\'ee
de $k$-vari\'et\'es, s'\'etendent 
en des morphismes $\W_{i} \subset \X_{{\bf m}_{i}} \subset \X$,
le premier \'etant une immersion ouverte de $\bf m_{i}$-sch\'emas, le second
une immersion ferm\'ee de $\calo_S$-sch\'emas.

(e) Le $k_{i}$-morphisme fini  \'etale $W'_{i} \to W_{i}$ s'\'etend en 
un morphisme fini \'etale surjectif  $\W'_{i} \to \W_{i} $  de ${\bf m}_{i}$-sch\'emas lisses,
 et les r\'esidus des
\'el\'ements de $B$ sont dans $H^1(\W'_{i} / \W_{i} , \Q/\Z) = H^1(K'_{i}/K_{i}, \Q/\Z)$.

(f) Pour tout $i$, toute place $w$ de $k_i$ non au-dessus d'une place de $S$
et totalement d\'ecompos\'ee dans $k'_i/k_i$, et tout $\sigma \in
\Gal(K'_i/k'_i K_i)$,
la fibre de $\W_{i}$ en $w$ poss\`ede
un $F(w)$-point dont le Frobenius (relativement au rev\^etement
$\W'_{i}/\W_{i}$) est $\sigma$. Ceci est possible via le lemme-cl\'e 
de \cite{ekedahl}. En particulier
pour toute place $v$ de $k$ non dans $S$ et tout $i$, la fibre de $\W_{i}$ au-dessus de tout
$F(v)$-point  de ${\bf m}_{i}$ poss\`ede un $F(v)$-point --
 mais il se peut que ni ${\bf m}_{i}$ ni  a fortiori $\W_{i}$   n'aient de $F(v)$-point.

(g) Tout \'el\'ement du groupe fini  $B'=B \cap  \Br X$ 
appartient \`a $\Br \X$.

\smallskip

On va maintenant ajouter \`a l'ensemble $B \subset \Br U$ de nouveaux
\'el\'ements. Appelons $\Lambda_1$ l'ensemble des \'el\'ements
de $\Br(k(t))$ non ramifi\'es en dehors de $\infty$ et des $m_i$, et
dont les r\'esidus en chaque $m_i$ sont dans $H^1(\Gal(k'_i/k_i),\Q/\Z)$.
D'apr\`es la suite exacte de Faddeev, le groupe $\Lambda_1/\Br k$
est fini.
Le groupe engendr\'e par les \'el\'ements $\cores_{k_i/k}(t-e_i,\chi_i)$, o\`u $e_i$ est vu dans 
$k_i$ et $\chi_i \in \Hom(\Gal(k'_i/k_i),\Q/\Z)$, est un sous-groupe fini
$\Lambda \subset \Lambda_1$ qui se surjecte sur $\Lambda_1/\Br k$.
Posons $C=B \cup \Lambda$ et $C'=C \cap \Br X$.
Quitte \`a agrandir encore $S$, on peut supposer $C' \subset \Br \X$.

\begin{lem}\label{nouvelletechnique}
Avec les hypoth\`eses et notations ci-dessus,
il existe un point ad\'elique
$$(P_v) \in  [X(k_{v_{0}})
\times \prod_{v \in S \setminus \{v_{0}  \} }   W_{v} \times \prod_{v \notin S} \U(O_{v})]^{C'}  \subset E$$
tel que de plus le point $P_{v_0}$ v\'erifie
\begin{equation} \label{nulalpha}
\alpha(P_{v_0})=0
\end{equation}
pour tout $\alpha \in B$ et
\begin{equation} \label{nulbeta}
\beta(f(P_{v_0}))=0
\end{equation}
pour $\beta \in \Lambda$.
\end{lem}

Noter que $\beta(f(P_{v_0}))=\beta(P_{v_0})$ 
puisque les \'el\'ements de $\Lambda$ sont dans $\Br(k(t))$.

\begin{proof} On commence d'abord
par choisir $\lambda_{v_0} \in U(k_{v_0})$ tel
que $\beta(\lambda_{v_0})=0$ pour tout $\beta \in \Lambda$. Pour $v_0$
complexe, n'importe quel $\lambda_{v_0}$ convient; pour $v_0$
r\'eel, il suffit de prendre $\lambda_{v_0}$ suffisamment grand; pour
$v_0$ fini, on cherche $\lambda_{v_0}$
dans $U(k_{v_0}) \subset k_{v_0}$ tel que 
$\lambda_{v_0}-e_i$ soit une puissance $N$-i\`eme dans $k_{w_0}$ pour tout $i$
et toute place $w_{0}$ de $k_{i}$ au-dessus de la place $v_{0}$,
o\`u $N$ est le ppcm des ordres des caract\`eres $\chi_i \in 
\Hom(\Gal(k'_i/k_i),\Q/\Z)$. Pour cela on prend $\lambda_{v_0}=1/p^{Nm}$ 
 o\`u $p$ est le nombre premier que divise 
$v_0$.   
On prend $m$ sufffisamment grand pour que
 pour tout $i$ et   tout compl\'et\'e de $k_{i}$
en une place $w_{0}$ au-dessus de la place $v_{0}$, $1-e_{i}p^{Nm}$
soit une puissance $N$-i\`eme dans $k_{i,w_{0}}$.

\smallskip

Posons alors $P_{v_0}=s_{v_0}(\lambda_{v_0})$.
Par construction on a alors $\alpha(P_{v_0})=0$ pour tout $\alpha \in B$
et $\beta(f(P_{v_0})=\beta(\lambda_{v_0})=0$ pour tout $\beta \in \Lambda$.

\smallskip

Par ailleurs on 
dispose d'un point ad\'elique $(Q_v)$ dans $E$, donc a
fortiori l'ensemble
$$E':= [X(k_{v_{0}})
\times \prod_{v \in S \setminus \{v_{0}  \} }   W_{v} \times \prod_{v \notin S} \X(O_{v})]^{C'}$$
est non vide.
Comme tous les \'el\'ements de $C$ s'annulent en $P_{v_{0}}$,
tout \'el\'ement du groupe $C'$ s'annule aussi en $P_{v_{0}}$.
Mais alors, comme tout \'el\'ement de $\Br X$ est (par l'hypoth\`ese iv))
constant quand \'evalu\'e sur $X(k_{v_{0}})$,
ceci implique que tout \'el\'ement de $C'$ s'annule sur $X(k_{v_{0}})$.
Pour $v\notin S$, on  a $\U(O_{v}) \neq \emptyset$ et tout
\'el\'ement de $C'$ s'annule 
   sur $\X(O_{v})$.
Pour toute place $v$, l'ensemble $U(k_{v} )$ est dense  dans $X(k_{v})$ pour la topologie $v$-adique.
  Il existe donc une famille $(P_{v}) \in U(\A_{k})^{C'}$
avec  $P_{v}
 \in U(k_{v}) \cap W_{v}$ pour toute place $v \in S \setminus \{v_{0}\}$, et $P_{v} \in \U(O_{v})$ pour $v\notin S$,
 donc $(P_{v})  \in U(\A_{k})$. Comme $C'$ s'annule en tout point de  $X(k_{v_{0}})$, on peut prendre pour $P_{v_{0}}$
 le point fix\'e plus haut.

\end{proof}

On fixe d\'esormais un point ad\'elique $(P_v)$ comme dans le lemme 
ci-dessus.
Le {\it lemme formel}  \cite[Lemme 2.6.1]{Hara} sous la forme donn\'ee dans \cite[d\'emonstration du Thm. 1.4]{CT},
 appliqu\'e \`a $U  \subset X$,
 assure alors l'existence d'un ensemble fini $T$ de places, $T \cap S=\emptyset$
 et  pour $v \in T$ de points $M_{v }\in U(k_{v}) \cap \X(O_{v})$  tels que
\begin{equation} \label{1star}
 \sum_{v\in S}\alpha(P_{v}) + \sum_{v \in T} \alpha(M_{v})= 0 \in \Q/\Z
\end{equation}
 pour tout $\alpha  \in C=B \cup \Lambda $.

Pour chaque  $i=1,\dots,l$, on choisit (ce qui est possible par le th\'eor\`eme
de Tchebotarev)
des places $w_i$ de $k_i$ de degr\'e absolu $1$,
ne divisant pas de places de $S \cup T$, totalement d\'ecompos\'ees
dans l'extension $k'_i/k_i$, et induisant des places $v_i$ de $k$
deux \`a deux distinctes.
On fixe  $\theta_i \in O_{v_i}$ tel que $v_i(P_i(\theta_i))=1$,
ce qui est possible car $w_i$ est de degr\'e $1$ sur $v_i$.

\medskip

On dispose donc des \'el\'ements $\lambda_{v}=f(P_{v}) \in k_{v}$ pour $v \in S \setminus{v_{0}}$,
$\lambda_{v}=f(M_{v}) \in O_{v}$ pour $v \in T$ et $\theta_{i}  \in O_{v_i}$ pour $i=1,\dots, l$.

Le th\'eor\`eme \ref{hilbertirr} combin\'e avec 
 la proposition \ref{appforteHilbertraf} appliqu\'ee \`a la r\'eunion
 de $S \setminus{\{v_{0}\}}$, $T$ et la r\'eunion des $\{v_{i}\}$,
fournissent alors  $\theta \in  \calo_S$, 
tr\`es proche de chaque $f(P_{v}) $ pour $v\in S\setminus \{v_{0}\}$,
de $f(M_{v}) \in O_{v}$ pour $v\in T$
et de $\theta_i \in O_{v_i}$ pour $i=1,\dots, l$ (on a donc
$v_i(P_i(\theta))=1$ pour tout $i$), avec de plus, 
en notant $\theta_{v_0}$ l'image de $\theta$ dans $k_{v_0}$~:

-- pour $v_0$ r\'eelle, $\theta_{v_0}$ suffisamment grand pour avoir
$\beta(\theta_{v_0})=0$ pour tout $\beta \in \Lambda$.

-- Pour $v_0$ finie, la propri\'et\'e
$\theta_{v_0}-e_i \in k_{i,w_0}^{*^N}$ pour toute place $w_{0}$ de $k_{i}$
au-dessus de $v_{0}$
pour $i=1,\dots,l$, o\`u $N$ est d\'efini comme le ppcm des ordres des $\chi_i
\in \Hom(\Gal(k'_i/k_i),\Q/\Z)$. En particulier on aura encore
\begin{equation} \label{betav0}
\beta(\theta_{v_0})=0 \in \Br k_{v_0}
\end{equation}
pour tout $\beta \in \Lambda$.

-- $\theta \in U_{0}(k)$ tel que {\it par sp\'ecialisation} en $\theta$
 l'image du groupe $B$ soit
 $\Br X_{\theta}/\Br k$.

\smallskip

On dispose maintenant du $\calo_S$-sch\'ema $\X_{\theta}$. Soit 
$X_{\theta}$ sa fibre g\'en\'erique. 

\begin{lem}\label{lemmeXtheta}
L'$\calo_S$-sch\'ema $\X_{\theta}$ poss\`ede~: 

\begin{itemize}

\item 
pour $v \in S$, un $k_v$-point $P'_v$, qui est de plus dans $W_v$ 
si $v \neq v_0$, 

\item 
pour $v \in T$, un $\calo_v$-point $M'_v$, 

\end{itemize}

\noindent
tels que pour tout  $\alpha \in C $ on ait
\begin{equation} \label{star3}
\sum_{v\in S} \alpha(P'_{v}) + \sum_{v \in T}\alpha(M'_{v})= 0.
\end{equation}

De plus $\X_{\theta}$ poss\`ede des points entiers 
$N_v \in {\mathcal X}(\calo_v)$ pour toute place $v \not \in 
S \cup T$.

\end{lem}

\begin{proof} Pour $v \in T \cup S \setminus \{v_{0} \}$, 
on utilise le th\'eor\`eme des
fonctions implicites. On obtient des points $P'_v$ sur $X_{\theta}$
arbitrairement proche des $P_v$ pour $v \in S \setminus \{v_0 \}$
(donc dans $W_v$) et des points $M'_v$ arbitrairement proche des
$M_v$ (donc entiers) pour $v \in T$. De  (\ref{1star}) on d\'eduit,
pour tout $\alpha \in C$,
\begin{equation} \label{star2}
\sum_{v\in S \setminus \{v_{0}\}} \alpha(P'_{v}) + \sum_{v \in T}\alpha(M'_{v})= 0,
\end{equation}
car $\alpha(P_{v_0})=0$ (formule~(\ref{nulalpha}))
 pour tout $\alpha \in B $ et $\beta(P_{v_0})=0$ (formule~(\ref{nulbeta}))
pour tout $\beta \in \Lambda$.
 Que l'on ait $X_{\theta}(k_{v_{0}})\neq \emptyset$ r\'esulte de l'existence d'une section $s_{v_{0}}$ de $X \to \A^1_{k}$
au-dessus de $U_{   0, k_{v_0}}  $.
 On note $P'_{v_{0}}=s_{v_{0}}(\theta) \in X_{\theta}(k_{v_{0}})$
le point donn\'e par la section.
Pour $\alpha \in B$, on a $$\alpha(P'_{v_{0}})=0$$
puisque $\alpha(s_{v_{0}}(t))=0$. Pour $\beta \in \Lambda$, on a
$$\beta(P'_{v_{0}})=\beta(\theta_{v_0})=0$$
par la formule~\ref{betav0}.

\smallskip 

La $k$-vari\'et\'e $X_{\theta}$ contient donc les points $P'_{v}, v \in S,$ et les points $M'_{v}, v \in T$,
et pour tout $\alpha \in B$ on a bien la formule (\ref{star3})~:
$$ \sum_{v\in S} \alpha(P'_{v}) + \sum_{v \in T}\alpha(M'_{v})= 0.$$

\smallskip

Soit maintenant $v$ une place non dans $S \cup T$.
 
 Si l'on a $v(P_{i}(\theta))=0$ pour tout $i$, alors en $v$,  $\theta$ se r\'eduit en un point de $\U_{0}(F(v))$,
qui est image d'un point de  $\X_{\theta}(F(v)) \subset \U(F(v))$. Par lissit\'e et Hensel, ce point se rel\`eve en un point  
$N_{v} \in \X_{\theta}(O_{v}) \subset \U(O_{v})$.
De plus tout \'el\'ement de $B$ s'annule 
sur un tel point,
 car 
 on a
  $B \subset \Br(\U)$.

Supposons $v(P_{i}(\theta))>0$ pour un $i$ (unique car les $P_{i}$ sont premiers entre eux deux \`a deux
en dehors de $S$). Alors $P_{i}$ admet un z\'ero sur $F(v)$, il existe une place $w$ de $k_{i}$
au-dessus de $v$ avec $F(w)=F(v)$. Soit $x \in \A^1_{\calo_S}$ le point ferm\'e de corps r\'esiduel $F(v)$
associ\'e. On a alors une inclusion ouverte $\W_{i,x} \subset \X_{x} = \X_{\theta,x}$.
 La $F(v)$-vari\'et\'e lisse $\W_{i,x}$ poss\`ede un $F(v)$-point $n_{v}$. Ainsi la $F(v)$-vari\'et\'e
$  \X_{\theta,x}$ poss\`ede le $F(v)$-point 
lisse
 $n_{v}$, qui se rel\`eve en un $O_{v}$-point $N_{v}$ du $\calo_S$-sch\'ema  
lisse
$\X_{\theta}$.
\end{proof} 

On appelle $\Omega_{0}$ (resp. $\Omega_i$)
l'ensemble des places $v \not \in S \cup T$
telles que $v(P_{i}(\theta))=0$ pour tout $i$ (resp. telles que
$v(P_{i}(\theta))>0$).
Pour $i$ fix\'e, l'ensemble $\Omega_i$ contient la place $v_i$.
L'ensemble $S \cup T$ (qui contient $v_{0} $) et  les ensembles $\Omega_{i}, i=0, 1,\dots, l$ r\'ealisent une partition de l'ensemble $\Omega$ des places
de $k$.

Pour $\alpha \in B $,  $v \notin S \cup T$ et $N_{v} \in \X_{\theta}(O_{v}) \subset X(k_{v})$ comme 
dans la d\'emonstration du lemme \ref{lemmeXtheta},  c'est-\`a-dire  relev\'e d'un point 
$n_{v} \in \W_{i,x}(F(v))$,
on a la formule \cite[Cor. 2.4.3 et pp.~244--245]{Hara}
\begin{equation} \label{star5}
\alpha(N_{v})= d_{i}(v) \partial_{\alpha,i}(F_{i,n_{v}}) \in \Q/\Z,
\end{equation}
o\`u $\partial_{\alpha,i} \in H^1(\W'_{i}/\W_{i},\Q/\Z) = \Hom(\Gal(K'_{i}/K_{i}),\Q/\Z)$ est le r\'esidu de $\alpha$ au point g\'en\'erique de $W_{i}$,
o\`u $d_{i}(v)=v(P_{i}(\theta))$ et $F_{i,n_{v}} \in \Gal(K'_{i}/K_{i})$ est le Frobenius en $n_{v}$
pour le  rev\^etement $\W'_{i}/\W_{i}$.

 \medskip

Fixons $i \in \{1,...,l\}$. On a alors le lemme suivant, analogue
du lemme 4 de \cite{Hara3}~:

\begin{lem}
L'\'el\'ement $\sum_{v \in \Omega_{i}} d_{i}(v) F_{i,n_{v}}$ de
$\Gal(K'_i/K_i)$ est dans le sous-groupe $\Gal(K'_i/k'_iK_i)$.
\end{lem}

\begin{proof}
On va utiliser l'\'egalit\'e (\ref{1star}) 
pour les \'el\'ements $\alpha \in \Lambda$
(c'est \`a cet effet qu'on a rajout\'e l'ensemble $\Lambda$ \`a 
$B$ pour obtenir l'ensemble $C$ auquel on a appliqu\'e le lemme 
formel). Soit 
$$\chi_i \in \Hom(\Gal(k'_i/k_i),\Q/\Z)=\Hom(\Gal(k'_iK_i/K_i),\Q/\Z),$$ 
appliquons 
(\ref{1star}) \`a l'\'el\'ement $\alpha_i=\cores_{k_i/k}(t-e_i,\chi_i)$ de 
$\Lambda$. D'apr\`es la loi de r\'eciprocit\'e du corps de classes 
global, on a, 
en notant $\theta_v$ l'image de $\theta$ dans $k_v$~:
$$\sum_{v \in \Omega_k} \alpha_i(\theta_v)=0$$
et donc d'apr\`es (\ref{star3}), comme les points $P'_v$ sont sur 
la fibre $X_{\theta}$~:
$$\sum_{v \not \in (S \cup T)} \alpha_i(\theta_v)=0.$$
On observe que si une place $v \not \in (S \cup T)$ n'est pas dans 
$\Omega_i$, alors $\alpha_i(\theta_v)=0$ par d\'efinition de $\Omega_i$. 
Ainsi 
$$\sum_{v \in \Omega_i} \alpha_i(\theta_v)=0$$
ce qui donne, par exemple en utilisant la formule (\ref{star5}),
$$\chi_i(\sum_{v \in \Omega_{i}} d_{i}(v) F_{i,n_{v}})=0$$
puisque $\alpha_i(\theta_v)=\alpha_i(N_v)$, l'\'el\'ement $\alpha_i$ 
\'etant dans $\Br(k(t))$. Finalement on a montr\'e que 
$\sum_{v \in \Omega_{i}} d_{i}(v) F_{i,n_{v}}$ \'etait annul\'e par 
tous les caract\`eres de $\Gal(k'_iK_i/K_i)$, d'o\`u le lemme.

\end{proof}

Par d\'efinition des places $w_i$ de $k_i$ et de la condition f) d\'efinie
avant le lemme~\ref{nouvelletechnique},
on peut alors gr\^ace au lemme pr\'ec\'edent
 trouver un point $r_i \in \W_{i}(F(v_i ))= \W_{i}(F(w_i))$
dont le Frobenius associ\'e  est
 $$F_{i,r_i }=F_{i,n_{v_i }} - \sum_{v \in \Omega_i}
d_{i}(v) F_{i,n_{v}} \in \Gal(K'_{i}/k'_iK_{i}) \subset \Gal(K'_i/K_i).$$
(Noter que comme $w_i$ est totalement d\'ecompos\'ee dans
l'extension $k'_i/k_i$, le Frobenius $F_{i,n_{v_i }}$ est aussi dans
$\Gal(K'_i/k'_i K_i)$).

\smallskip

On rel\`eve $r_i  \in \W_{i}(F(v _i  )) \subset \X_{\theta}(F(v _i  ))$
  en un point $Q_i \in \X_{\theta}(O_{v_i })$.
On a $$d_{i}(v_i )=v_i (P_i(\theta))=v_i(P(\theta_i ))=1.$$

On a alors pour tout $\alpha \in B$ :
\begin{equation} \label{star4}
\alpha(Q_i)  + \sum_{v \in \Omega_i \setminus \{v_i \}} \alpha(N_{v})
= \partial_{\alpha,i}(F_{i,r_i}-F_{i,n_{v_{i} }} + \sum_{v \in \Omega_{i}} d_{i}(v) F_{i,n_{v}})=
\partial_{\alpha,i}(0)=0.
\end{equation}
Tout ad\`ele de $X_{\theta}$ de composantes $P'_{v_{0}}$ en $v_{0}$,  $P'_{v}, v \in S \setminus v_{0}$, $M'_{v} , v \in T$,
 $Q_{i}$ en $v_{i}$ et $N_{v} $ pour $v \in \Omega_{i} \setminus \{v_{i}\}$ ($i=1,\dots, l$), et enfin  $N_{v}$ pour les autres places $v$
est alors dans
$$  \prod_{v \in S} X_{\theta}(k_{v}) \times \prod_{v \notin S} \X_{\theta}(O_{v})$$
et d'apr\`es (\ref{star3}) et (\ref{star4}) est orthogonal
\`a tout $\alpha \in B$, donc \`a $\Br X_{\theta}$ -- puisque $B$ se surjecte sur  $\Br X_{\theta}/\Br k$ . En outre il appartient
\`a
$$X(k_{v_{0}}) \times \prod_{v \in S \setminus \{v_{0}  \} }   W_{v} \times \prod_{v \notin S} \X(O_{v}).$$
 
L'un des principaux r\'esultats de  \cite{CTX}, le
 th\'eor\`eme 3.7,  appliqu\'e \`a 
l'espace homog\`ene  $\X_{\theta}$,
  montre alors qu'il existe un point de $X_{\theta}(k)$
 proche de chacun des $P'_{v}, v \in S \setminus v_{0}$, donc dans $W_v$ 
pour $v \in S \setminus v_{0}$,
 et entier en dehors de $S$ sur $\X_{\theta}$, donc sur $\X$. Ceci 
ach\`eve la preuve du th\'eor\`eme.
\end{proof}

\begin{rem}\label{souventsatisfaite}
 {\rm

  L'hypoth\`ese (iv)  est satisfaite dans chacun des cas suivants :

\medskip

(a) 
 $\Br X/\Br k = 0$; dans ce cas la conclusion du th\'eor\`eme est que 
$X$ v\'erifie l'approximation forte en dehors de $v_0$. 
 
 (b)
$\Br X_{k_{v_{0}}} /\Br k_{v_{0}}=0$;

(c) La  place  $v_{0}$ est complexe; dans ce cas l'hypoth\`ese (iii) est toujours satisfaite, car $k_{v_{0}}(t)$
est alors un corps $C_{1}$ par le th\'eor\`eme de Tsen (\cite{gillesza}, 
Th. 6.2.8.). Sur un tel corps, c'est un th\'eor\`eme de Springer que tout
  espace homog\`ene
d'un groupe lin\'eaire connexe a un point rationnel \cite[Chap. III,  \S 2, Thm. 1' et Cor. 2]{CG}.
}

\end{rem}

\begin{rem}
{\rm On peut noter que la d\'emonstration ci-dessus est un raffinement 
de la d\'emonstration du
 th\'eor\`eme 2 de \cite{Hara3}.  
On a remplac\'e l'ensemble infini $\Sigma$ de places de $k$
 et la place appel\'ee $v_{\infty} $ dans  \cite{Hara3}  par
 l'unique place $v_{0}$ de l'\'enonc\'e ci-dessus.
L'hypoth\`ese d'existence d'une $k_{v_0}$-section rationnelle 
permet de contr\^oler ce qui se passe en la place $v_0$; ceci remplace le 
travail avec la place $v_{\infty}$ et l'ensemble infini 
$\Sigma$ de \cite{Hara3}. On a aussi eu besoin d'une version 
sophistiqu\'ee du th\'eor\`eme 
d'approximation forte (proposition~\ref{appforteHilbertraf})}
\end{rem}

\section{L'\'equation $\sum_{i=0}^2 a_{i}(t)x_{i}^2=p(t).$ }

Soit $k$ un corps de caract\'eristique z\'ero, $\k$ une cl\^oture alg\'ebrique.
Soient $a_{i}(t), i=0,1,2,3$ dans $k[t]$. 
{\it On suppose que les $a_{i}(t)$ sont premiers entre eux deux \`a deux.}\medskip

Soit $Y_{0} \subset \A^1_{k} \times \P^3_{k}$ la $k$-vari\'et\'e d\'efinie par
$$ \sum_{i=0}^3 a_{i}(t) x_{i}^2=0.$$

Les points
 singuliers de $Y_{0}$ sont exactement les points 
dont la coordonn\'ee $t$ est une racine multiple de l'un des
$a_{i}(t)$ et tels que pour cet $i$, $x_{i}=1$ et $x_{j}=0$ pour $j\neq i$.
 
Notons $Y \subset Y_{0}$ l'ouvert de lissit\'e, compl\'ement de ces points.

Soit $Z_{0} \subset Y_{0}$ le ferm\'e d\'efini par $x_{3}= 0$. C'est donc la $k$-vari\'et\'e
d\'efinie dans $\A^1_{k} \times \P^2_{k}$ par 
$$ \sum_{i=0}^2 a_{i}(t) x_{i}^2=0.$$

Les point singuliers de  $Z_{0}$ sont
les points 
dont la coordonn\'ee $t$ est une racine multiple de l'un des
$a_{i}(t)$ et tels que pour cet $i$, $x_{i}=1$ et $x_{j}=0$ pour $j\neq i$.
Soit 
$Z \subset Z_{0}$ le lieu de lissit\'e.
Soit $X \subset \A^1_{k} \times \P^3_{k}$ le lieu de lissit\'e de la vari\'et\'e d\'efinie par
$$ \sum_{i=0}^2 a_{i}(t) x_{i}^2 + a_{3}(t)=0.$$

On a donc $Z  \subset Y$, et $X$ est le compl\'ementaire de
$Z$ dans $Y$.

\medskip

L'\'enonc\'e suivant g\'en\'eralise  \cite[Prop. 5.2]{CTX2}.

\begin{prop}\label{brauernul} Soit $X$ comme ci-dessus.
Si  le produit des $a_{i}$ n'est pas un carr\'e dans $\k[t]$,
alors $\Br X/\Br k=0.$
\end{prop}

\begin{proof}
On note $K=\k(t)$.
Comme $Y$ et $Z$ sont lisses, on a 
la suite exacte :
$$ 0 \to \Br \ov{Y}  \to \Br \ov{X} \to H^1_{\et}(\ov{Z},\Q/\Z).$$
(cf. \cite{groth}, III, Cor. 6.2.)

En se restreignant au-dessus de $\Spec K$ et en utilisant le fait
que $Z \to \A^1_{k}$ a pour  fibre g\'en\'erique une conique lisse et a
toutes ses fibres g\'eom\'etriques non vides
et de multiplicit\'e 1, on \'etablit
$$H^1_{\et}(\ov{Z},\Q/\Z)=0.$$

On a ensuite l'inclusion
$ \Br \ov{Y} \hookrightarrow  \Br Y_{K}$. La $K$-vari\'et\'e  $Y_{K}$ est une quadrique projective
et lisse de dimension 2 sur $K=\k(t)$. On conclut
$\Br Y_{K}=0$.
Ainsi  $\Br \ov{Y}=0$ et $ \Br \ov{X}=0$.

Comme le discriminant de la quadrique $Y_{K}$ n'est pas un carr\'e, 
on a $\pic Y_{K}=\Z$ et $\pic X_{K}=0$. Comme les fibres de $\ov{X} \to \A^1_{\k}$
sont int\`egres, on en d\'eduit $\pic \ov{X}=0$.

On v\'erifie facilement que l'on a $\k^{\times}=\k[X]^{\times}$.
Soit $g=\Gal(\k/k)$. De la suite exacte
$$ \pic \ov{X}^g \to H^2(g,\k[X])^{\times}) \to \ker[\Br X \to \Br \ov{X}]    \to H^1(g,\pic \ov{X})$$
on d\'eduit
$$\Br k \oi \Br X.$$
\end{proof}

Le th\'eor\`eme suivant  couvre les principaux \'enonc\'es
du  th\'eor\`eme 6.7 de \cite{CTX2}.

\begin{theo}\label{generalCTX2}
 Soit $k$ un corps de nombres.
 Soient $a_{i}(t)$ et $p(t)$ dans $k[t]$ des polyn\^omes deux \`a deux premiers entre eux.
  Soit $X$ l'ouvert de lissit\'e de la $k$-vari\'et\'e affine d\'efinie
 dans $\A^4_{k}$ par l'\'equation
$$\sum_{i=0}^2 a_{i}(t)x_{i}^2=p(t).$$

Supposons $X(\A_{k}) \neq \emptyset$.
Soit $v_{0}$ une place de $k$  telle que la forme quadratique diagonale
$ <a_{1}(t),  a_{2}(t), a_{3}(t) >$
ait un z\'ero sur $k_{v_{0}}(t)$. Alors~:

\smallskip

(a) Si $v_0$ est complexe,
l'image diagonale de $X(k)$ dans la projection de $X(\A_k)^{\Br X}$
sur $X(\A_k^{v_0})$ est dense.

\smallskip

(b) La place $v_0$ \'etant \`a nouveau suppos\'ee quelconque,
si le produit $p(t).\prod_{i=0}^2a_{i}(t)$ n'est pas un carr\'e dans $\k[t]$,
la $k$-vari\'et\'e $X$ satisfait l'approximation forte hors de $v_{0}$ :
l'image diagonale de $X(k)$ dans  $X(\A_k^{v_0})$
(ad\`eles hors de $v_{0}$) est dense.

 \end{theo}

 \begin{proof}
 Montrons que les hypoth\`eses (i) \`a (iv)  du th\'eor\`eme \ref{thprincipal} sont satisfaites.
 
  La fibre g\'en\'erique de $f : X \to \A^1_{k}$
 est une quadrique affine, espace homog\`ene sous le groupe $G_{t}$ des spineurs
 de la forme quadratique diagonale 
 $ <a_{1}(t), a_{2}(t), a_{3}(t) >$, les stabilisateurs g\'eom\'etriques \'etant
 des tores de dimension 1  (voir \cite[\S 5.6 et \S 5.8]{CTX}). On a donc l'hypoth\`ese (i).
 
 L'hypoth\`ese de coprimalit\'e faite sur les $a_{i}(t)$ et $p(t)$ assure que toutes les fibres de $X \to \A^1_{k}$
 sont  g\'eom\'etriquement int\`egres, et donc
 scind\'ees, ce qui donne (ii).

  L'hypoth\`ese faite sur la place $v_{0}$ implique 
  qu'en tout $t_{0} \in k_{v_{0}}$
  o\`u le produit $p(t).\prod_{i=0}^2a_{i}(t)$ n'est pas nul, la forme quadratique
   $ <a_{1}(t_{0}), a_{2}(t_{0}), a_{3}(t_{0}) >$ est isotrope,
   et ceci garantit que le $k_{v_{0}}$-groupe $G_{t_{0}}$ est isotrope, ce qui donne (iii). 

Pour $v_0$ complexe, la condition (iv) est \'evidente. Pour $v_0$
quelconque mais avec l'hypoth\`ese suppl\'ementaire que
$p(t).\prod_{i=0}^2a_{i}(t)$ n'est pas un carr\'e dans $\k[t]$,
on a $\Br X/\Br k=0$ par la proposition \ref{brauernul}
donc (iv) est encore v\'erifi\'ee.

\smallskip

   On voit donc que les hypoth\`eses (i) \`a (iv) du th\'eor\`eme \ref{thprincipal} sont satisfaites.
   Ceci \'etablit l'\'enonc\'e (a).

   L'\'enonc\'e (b) r\'esulte alors de loc. cit.
   
 \end{proof}

  \begin{rem}\label{witt}
 {\rm 
Si $v_{0}$ est une place r\'eelle,  c'est un th\'eor\`eme de Witt qu'une forme quadratique
de rang au moins 3 sur $k_{v_{0}}(t)$ a un z\'ero si et seulement si pour presque tout
$t_{0} \in k_{v_{0}}$ la forme quadratique sp\'ecialis\'ee a un z\'ero.}
 \end{rem}

\begin{rem}
{\rm 
Dans \cite{CTX2}, on consid\`ere la $k$-vari\'et\'e $X_{0}$ donn\'ee par l'\'equation 
 $$ \sum_{i=0}^2 a_{i}x_{i}^2=p(t),$$
 avec les $a_{i} \in k^{\times}$ et $p(t) \in k[t]$ non nul. Soit $X$ son ouvert de lissit\'e.
 On autorise le polyn\^ome $p(t)$ \`a \^etre un carr\'e dans $\k[t]$.
 Avec les notations ci-dessus, on \'etablit
que  l'image diagonale de $X(k)$ dans la projection de l'ensemble de Brauer--Manin 
$X(\A_k)^{\Br X}$
sur $X(\A_k^{v_0})$ (ad\`eles hors de $v_{0}$) est dense.

Lorsque $p(t)$ est un carr\'e dans $\k[t]$, on peut avoir $\Br X/\Br k=\Z/2$.
Dans \cite[Thm. 6.4, Thm. 6.5]{CTX2}, on donne aussi des \'enonc\'es pour
l'approximation forte sur une r\'esolution non singuli\`ere de $X_{0}$,
c'est-\`a-dire une $k$-vari\'et\'e lisse $Z$ \'equip\'ee d'un $k$-morphisme
propre et birationnel $Z \to X_{0}$. Le th\'eor\`eme principal du pr\'esent article
permet aussi de traiter ce cas quand $v_0$ est complexe, mais pas le cas
o\`u $v_0$ est quelconque.
}
\end{rem}

{\bf Remerciements. } Les auteurs tiennent \`a remercier 
chaleureusement Dasheng Wei
pour leur avoir signal\'e une erreur dans la premi\`ere version de cet
article, et de les avoir aid\'es 
\`a la corriger.

\end{document}